\documentclass[reqno]{amsproc}
\usepackage{amsmath}
\usepackage{amssymb}
\usepackage{euscript} %txfonts,pxfonts
\usepackage{graphicx}
\usepackage{amsfonts}
\usepackage{mathrsfs}
\usepackage{amssymb}
\usepackage[all]{xy}
\usepackage{color}
\usepackage{mathrsfs} %\mathscr T
\usepackage{color}

\makeatletter
\@namedef{subjclassname@2010}{%
\textup{2010} Mathematics Subject Classification}
\makeatother

\numberwithin{equation}{section}
\newtheorem{thm}{Theorem}[section]
\newtheorem{cor}[thm]{Corollary}
\newtheorem{lem}[thm]{Lemma}
\newtheorem{pro}[thm]{Proposition}

\newtheorem*{thm*}{Theorem}

\newtheorem*{opq*}{Problem}

\theoremstyle{remark}
\newtheorem{rem}[thm]{Remark}

\theoremstyle{definition}
\newtheorem{exa}[thm]{Example}

\newtheorem{dfn}[thm]{Definition}

\DeclareMathOperator{\dzii}{{\mathsf{Chi}}}

\newcommand*{\card}[1]{\mathrm{card}(#1)}
\newcommand*{\cbb}{\mathbb C}

\newcommand*{\dzi}[1]{\dzii(#1)}

\newcommand*{\Ge}{\geqslant}
\newcommand*{\hh}{\mathcal H}

\newcommand*{\lambdab}{\boldsymbol \lambda}
\newcommand*{\lambdabh}{\hat{\lambdab}}
\newcommand*{\Le}{\leqslant}

\newcommand*{\ogr}[1]{{\boldsymbol B (#1)}}
\newcommand*{\rbb}{\mathbb R}
\newcommand*{\slam}{{S_{\lambdab}}}
\newcommand*{\slamh}{S_{\hat{\lambdab}}}
\newcommand*{\supp}[1]{\mathrm{supp}\,#1}
\newcommand*{\tcal}{{\mathscr T}}
\newcommand*{\tcalh}{\hat{\mathscr T}}
\newcommand*{\zbb}{\mathbb{Z}}
\newcommand*{\nbbc}{\bar{\mathbb N}_{2}}

\newcommand*{\nbb}{\mathbb{N}}

\begin{document}

   \title[A Subnormal Completion Problem, II]{A
Subnormal Completion Problem for \\ Weighted Shifts on
Directed Trees, II}

   \author[G. R. Exner]{George R. Exner}
   \address{Department of Mathematics, Bucknell
University, Lewisburg, Pennsylvania 17837, USA}
\email{exner@bucknell.edu}
   \author[I.\ B.\ Jung]{Il Bong Jung}
   \address{Department of Mathematics,
Kyungpook National University, Da\-egu 41566, Korea}
\email{ibjung@knu.ac.kr}
   \author[J.\ Stochel]{Jan Stochel}
\address{Instytut Matematyki, Uniwersytet
Jagiello\'nski, ul.\ \L ojasiewicza 6, PL-30348 Kra\-k\'ow,
Poland} \email{Jan.Stochel@im.uj.edu.pl}
   \author[H. Y. Yun]{Hye Yeong Yun}
   \address{Department of Mathematics,
Kyungpook National University, Da\-egu 41566, Korea}
\email{yunhy@knu.ac.kr}



   \thanks{The research of the second author was supported by the National Research
Foundation of Korea (NRF) grant funded by the Korea
government (MSIT) (2018R1A2B6003660).}

\date{}
\maketitle

   \begin{abstract}
The subnormal completion problem on a directed tree is
to determine, given a collection of weights on a
subtree, whether the weights may be completed to the
weights of a subnormal weighted shift on the directed
tree. We study this problem on a directed tree with a
single branching point, $\eta$ branches and the trunk
of length $1$ and its subtree which is the
``truncation'' of the full tree to vertices of
generation not exceeding $2$. We provide necessary and
sufficient conditions written in terms of two
parameter sequences for the existence of a subnormal
completion in which the resulting measures are
$2$-atomic. As a consequence, we obtain a solution of
the subnormal completion problem for this pair of
directed trees when $\eta < \infty$. If $\eta=2$, we
present a solution written explicitly in terms of
initial data.
   \end{abstract}
   \section{Introduction}
The class of unilateral weighted shifts on Hilbert
space has been a standard and important source of
examples with which to study the properties of bounded
linear operators on Hilbert space, including
especially the investigation of subnormality (see
\cite{Shi} and \cite{Con1991}). A recently introduced
class of weighted shifts on directed trees provides a
more extensive collection of objects for study (see
e.g., \cite{JJS,JJS2012,BJJS2012,BJJS2013,BJJS2018}).
In \cite{EJSY1}, we initiated the study of a subnormal
completion problem for weighted shift operators on
directed trees. For a classical weighted shift, the
subnormal completion problem is to be given an initial
finite sequence of positive weights and to determine
whether or not they may be extended to the weights of
an injective, bounded, subnormal unilateral weighted
shift; such a shift is called a subnormal completion
of the initial weight sequence (see
\cite{Sta,JJKS2011}; see also
\cite{CuF91,CuF93,CuF94,CuF1996}). We consider the
analogous task in the setting of weighted shifts on
directed trees.

In the present paper we continue the study of the
subnormal completion problem for weighted shifts on
directed trees. As in \cite{EJSY1}, we restrict our
attention to the directed tree $\tcal_{\eta,\kappa}$
with a single branching point, $\eta$ branches and the
trunk of length $\kappa$, and we consider the
subnormal completion problem with respect to the
subtree $\tcal_{\eta,\kappa,p}$, which is the
``truncation'' of $\tcal_{\eta,\kappa}$ to vertices of
generation not exceeding $p$. If $\eta$ is finite and
the $p$-generation subnormal completion problem on
$\tcal_{\eta,\kappa}$ has a solution for given initial
data $\lambdab$, then we can always find a subnormal
completion $\slamh$ such that the measures
$\mu_{i,1}^{\lambdabh}$, which are canonically
associated with $\slamh$ at vertices of the first
generation, are at most
$\lfloor\frac{\kappa+p+2}{2}\rfloor$-atomic (see
Theorem~\ref{finite-su}). So far as we know, there is
no solution of the $p$-generation subnormal completion
problem on $\tcal_{\eta,\kappa}$ written in terms of
initial data for $p\Ge 2$. The only explicit solution
is that for the $1$-generation subnormal completion
problem on $\tcal_{\eta,1}$ (see
\cite[Theorem~5.1]{EJSY1}). In view of the above
discussion, if the $2$-generation subnormal completion
problem on $\tcal_{\eta,1}$ has a solution for given
initial data $\lambdab$, then we can always find a
subnormal completion $\slamh$ of $\lambdab$ on
$\tcal_{\eta,1}$ with the property that each measure
$\mu_{i,1}^{\lambdabh}$ is $1$- or $2$-atomic. In this
paper we provide necessary and sufficient conditions
written in terms of two parameter sequences for
$\lambdab$ to admit a subnormal completion $\slamh$ on
$\tcal_{\eta,1}$ with $2$-atomic measures
$\mu_{i,1}^{\lambdabh}$. As a consequence, we solve
the $2$-generation subnormal completion problem on
$\tcal_{\eta,1}$ for $\eta <\infty$. The results,
taken as a whole, suggest that a complete answer to
the subnormal completion problem even for this class
of directed trees is at present out of reach.

The paper is organized as follows. In Section~ \ref{Sect2},
we provide notation, terminology and results that are
needed in this paper. In Section~ \ref{Sect3} we carry out
an in-depth analysis of the first two ``negative'' moments
of the Berger measure associated with the Stampfli
completion of three increasing weights to the weight
sequence for a subnormal unilateral weighted shift (see
Lemma~ \ref{prop:formsoftheints}). In Section~\ref{Sect4}
we state and prove the solutions of the $2$-generation
subnormal completion problem on $\tcal_{\eta,1}$ with
$2$-atomic measures (the case $\eta=\infty$ is included,
see Theorem~\ref{charoftwoatomic-infty}) and without any
restrictions on supports of the associated measures (only
for $\eta < \infty$; see
Theorem~\ref{charoftwoatomic-solution}). In view of
Proposition~\ref{charoftwoatomic}(iii), to solve the
problem with $2$-atomic measures we have to compute the
infimum of a quadratic form in $\eta$ variables subject to
some constraints. This task is extremely complicated. In
Section~\ref{Sect5} we give a solution of the
$2$-generation subnormal completion problem on
$\tcal_{2,1}$ with $2$-atomic measures written entirely in
terms of the initial data (see Theorem~\ref{iff-eta2}).
This confirms the scale of the complexity of this problem.
   \section{\label{Sect2}Preliminaries}
In this section we sketch briefly the notation and results
necessary for the present discussion, but the reader is
encouraged to consult \cite{EJSY1} for a considerably more
complete presentation.

Given a complex Hilbert space $\hh$, we denote by
$\ogr{\hh}$ the $C^*$-algebra of all bounded linear
operators on $\hh$. Recall that an operator $T\in
\ogr{\hh}$ is said to be {\em subnormal} if there
exists a complex Hilbert space $\mathcal{K}$
containing $\hh$ and a normal operator $N\in
\ogr{\mathcal{K}}$ such that $Th = Nh$ for all $h\in
\hh$. We refer the reader to
\cite{Con1981,Con1991,Hal1950} for the foundations of
the theory of subnormal operators.

Denote by $\zbb_+$, $\nbb$, $\rbb$, $\rbb_+$ and $\cbb$ the
sets of nonnegative integers, positive integers, real
numbers, nonnegative real numbers and complex numbers,
respectively. Set $\nbb_2 = \{2,3,4, \ldots\}$ and
$\nbbc=\nbb_2 \cup \{\infty\}$. Define $J_{\iota }$ by
   \begin{align*}
\text{$J_{\iota }=\{k\in \mathbb{N}\colon k\Le \iota \}$
for $\iota \in \zbb_+ \cup \{\infty\}$},
   \end{align*}
using the convention that $J_0=\varnothing$. We write
$\card{X}$ for the cardinality of a set $X$. In what
follows, $\delta_x$ stands for the Dirac Borel measure
on $\rbb_+$ at the point $x\in \rbb_+$. The closed
support of a Borel probability measure $\mu$ on
$\rbb_+$ is denoted by $\supp{\mu}$.

A pair $\mathscr{G}=(V,E)$ is a \textit{directed
graph} if $V$ is a nonempty set and $E$ is a subset of
$V\times V$. A member of $V$ is a \textit{vertex\ of
}$\mathscr{G}$, and an element of $E$ is called an
\textit{edge} of $\mathscr{G}$. For $u\in V$, we set
$\dzi{u}=\{v\in V:(u,v)\in E\}$. A member of $\dzi{u}$
is called a \textit{child} of $u$. A vertex $v$ of
$\mathscr{G}$ is called a \textit{root} of
$\mathscr{G}$, which we may also write as $v\in
\mathrm{Root}(\mathscr{G}) $, if there is no vertex
$u$ of $\mathscr{G}$ such that $(u,v)$ is an edge of
$\mathscr{G}$. We set $V^{\circ }=V\setminus
\mathrm{Root}(\mathscr{G})$. We say that $\tcal=(V,E)$
is a {\em directed tree} if $\tcal$ is a directed
graph, which is connected, has no circuits and has the
property that for each vertex $u \in V$ there exists
at most one vertex $v\in V$, called the {\em parent}
of $u$ and denoted here by $\mathrm{par}(u)$, such
that $(v,u)\in E$. The reader is referred to
\cite{JJS,Ore} for more information on directed trees
needed in this paper.
   \begin{figure}[ht]
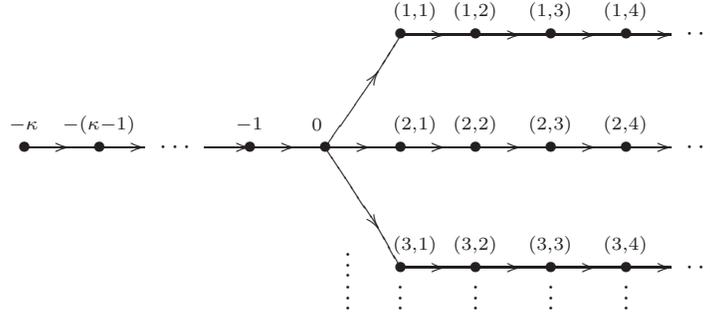

\centerline{\xy (0,8);(6,8)**@{-}?>*\dir{>},
(6,8);(10,8)**@{-}, (10,8);(16,8)**@{-}?>*\dir{>},
(24,8);(30,8)**@{-}?>*\dir{>}, (30,8);(36,8)**@{-}
?>*\dir{>}, (36,8);(40,8)**@{-},
(40,8);(47,18)**@{-}?>*\dir{>},
(47,18);(50,22.8)**@{-},
(40,8);(46,8)**@{-}?>*\dir{>}, (46,8);(50,8)**@{-},
(50,8);(56,8)**@{-}?>*\dir{>}, (56,8);(60,8)**@{-},
(60,8);(66,8)**@{-}?>*\dir{>}, (66,8);(70,8)**@{-},
(70,8);(76,8)**@{-}?>*\dir{>}, (76,8);(80,8)**@{-},
(80,8);(86,8)**@{-}?>*\dir{>},(86,8);(90,8);
(50,23);(56,23)**@{-}?>*\dir{>},
(56,23);(60,23)**@{-},
(60,23);(66,23)**@{-}?>*\dir{>},
(66,23);(70,23)**@{-},
(70,23);(76,23)**@{-}?>*\dir{>},
(76,23);(80,23)**@{-},
(80,23);(86,23)**@{-}?>*\dir{>},(86,23);
(90,23);(40,8);(47,-2.66)**@{-}?>*\dir{>},
(47,-2.66);(50,-8)**@{-},
(50,-8);(56,-8)**@{-}?>*\dir{>},
(56,-8);(60,-8)**@{-},
(60,-8);(66,-8)**@{-}?>*\dir{>},
(66,-8);(70,-8)**@{-},
(70,-8);(76,-8)**@{-}?>*\dir{>},
(76,-8);(80,-8)**@{-},
(80,-8);(86,-8)**@{-}?>*\dir{>}, (86,-8);
(90,-8);(0,8)*{\bullet}, (10,8)*{\bullet},
(0,11)*{_{-\kappa}}, (10,11)*{_{-(\kappa-1)}},
(30,11)*{_{-1}}, (39,11)*{_{0}}, (52,26)*{_{(1,1)}},
(52,11)*{_{(2,1)}}, (52,-5)*{_{(3,1)}},
(60,-5)*{_{(3,2)}}, (70,-5)*{_{(3,3)}},
(80,-5)*{_{(3,4)}}, (60,11)*{_{(2,2)}},
(70,11)*{_{(2,3)}}, (80,11)*{_{(2,4)}},
(60,26)*{_{(1,2)}}, (70,26)*{_{(1,3)}},
(80,26)*{_{(1,4)}}, (50,23)*{\bullet},
(60,23)*{\bullet}, (70,23)*{\bullet},
(80,23)*{\bullet}, (30,8)*{\bullet}, (40,8)*{\bullet},
(50,8)*{\bullet}, (60,8)*{\bullet}, (70,8)*{\bullet},
(80,8)*{\bullet}, (90,-8)*{\cdots}, (90,8)*{\cdots},
(20,8)*{\cdots}, (90,23)*{\cdots}, (43,-6.6)*{\vdots},
(43,-11)*{\vdots}, (50,-11)*{\vdots},
(60,-11)*{\vdots}, (70,-11)*{\vdots},
(80,-11)*{\vdots}, (50,-8)*{\bullet},
(60,-8)*{\bullet}, (70,-8)*{\bullet},
(80,-8)*{\bullet},
\endxy}
\vspace*{5pt} \caption{An illustration of the directed
tree $\tcal_{\eta,\kappa}$.\label{Fig0}} \centering
   \end{figure}

We shall consider a certain class of directed trees with a
single branching point obtained as follows: given $\eta \in
\nbbc$ and $\kappa \in \nbb$, we define the directed tree
$\tcal_{\eta ,\kappa }=(V_{\eta,\kappa}, E_{\eta,\kappa})$
by (see Figure \ref{Fig0}):
   \allowdisplaybreaks
   \begin{align*}
V_{\eta ,\kappa } &=\{-k:k\in J_{\kappa }\}\sqcup
\{0\}\sqcup \{(i,j):i\in J_{\eta },j\in \mathbb{N}\},
   \\
E_{\eta ,\kappa } &=E_{\kappa }\sqcup \{(0,(i,1)):i\in
J_{\eta }\}\sqcup \{((i,j),(i,j+1)):i\in J_{\eta
},j\in \mathbb{N}\},
   \\
E_{\kappa}&=\{(-k,-k+1)\colon k\in J_{\kappa }\}.
   \end{align*}
Define as well a subtree of $\tcal_{\eta, \kappa}$ on which
we may be given ``some of'' the weights of a proposed shift
as initial data: for $\eta \in \nbbc$, $\kappa \in \nbb$
and $p \in \nbb$, let the directed tree $\tcal_{\eta,
\kappa, p}=(V_{\eta, \kappa, p},E_{\eta ,\kappa ,p})$ be
defined by (see Figure \ref{Fig0.1})
   \begin{align*}
V_{\eta ,\kappa ,p} &=\{-k:k\in J_{\kappa }\}\sqcup
\{0\}\sqcup \{(i,j):i\in J_{\eta },j\in J_{p}\},
   \\
E_{\eta ,\kappa ,p} &=E_{\kappa }\sqcup
\{(0,(i,1)):i\in J_{\eta }\}\sqcup
\{((i,j),(i,j+1)):i\in J_{\eta },j\in J_{p-1}\}.
   \end{align*}
   \begin{figure}[ht]
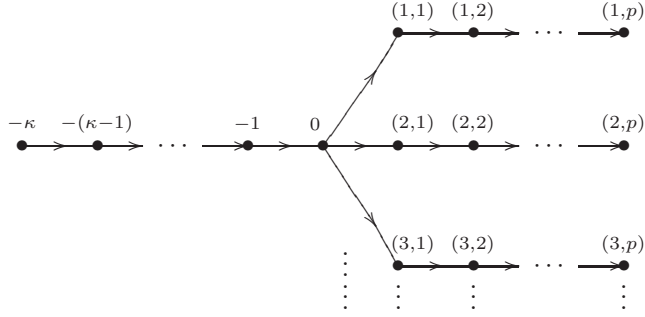

\centerline{\xy (0,8);(6,8)**@{-}?>*\dir{>},
(6,8);(10,8)**@{-}, (10,8);(16,8)**@{-}?>*\dir{>},
(24,8);(30,8)**@{-}?>*\dir{>}, (30,8);(36,8)**@{-}
?>*\dir{>}, (36,8);(40,8)**@{-},
(40,8);(47,18)**@{-}?>*\dir{>},
(47,18);(50,22.8)**@{-},
(40,8);(46,8)**@{-}?>*\dir{>}, (46,8);(50,8)**@{-},
(50,8);(56,8)**@{-}?>*\dir{>}, (56,8);(60,8)**@{-},
(60,8);(66,8)**@{-}?>*\dir{>},
(74,8);(80,8)**@{-}?>*\dir{>}, (80,8);(86,8);(90,8);
(50,23);(56,23)**@{-}?>*\dir{>},
(56,23);(60,23)**@{-},
(60,23);(66,23)**@{-}?>*\dir{>}, (70,23)*{\cdots},
(74,23);(80,23)**@{-}?>*\dir{>}, (80,23);(86,23);
(90,23);(40,8);(47,-2.66)**@{-}?>*\dir{>},
(47,-2.66);(50,-8)**@{-},
(50,-8);(56,-8)**@{-}?>*\dir{>},
(56,-8);(60,-8)**@{-},
(60,-8);(66,-8)**@{-}?>*\dir{>}, (70,-8)*{\cdots},
(74,-8);(80,-8)**@{-}?>*\dir{>}, (80,-8); (86,-8);
(90,-8);(0,8)*{\bullet}, (10,8)*{\bullet},
(0,11)*{_{-\kappa}}, (10,11)*{_{-(\kappa-1)}},
(30,11)*{_{-1}}, (39,11)*{_{0}}, (52,26)*{_{(1,1)}},
(52,11)*{_{(2,1)}}, (52,-5)*{_{(3,1)}},
(60,-5)*{_{(3,2)}}, (80,-5)*{_{(3,p)}},
(60,11)*{_{(2,2)}}, (80,11)*{_{(2,p)}},
(60,26)*{_{(1,2)}}, (80,26)*{_{(1,p)}},
(50,23)*{\bullet}, (60,23)*{\bullet},
(80,23)*{\bullet}, (30,8)*{\bullet}, (40,8)*{\bullet},
(50,8)*{\bullet}, (60,8)*{\bullet}, (80,8)*{\bullet},
(70,8)*{\cdots}, (20,8)*{\cdots}, (43,-6.6)*{\vdots},
(43,-11)*{\vdots}, (50,-11)*{\vdots},
(60,-11)*{\vdots}, (80,-11)*{\vdots},
(50,-8)*{\bullet}, (60,-8)*{\bullet},
(80,-8)*{\bullet},
\endxy}
\vspace*{5pt} \caption{An illustration of the directed
tree $\tcal_{\eta,\kappa,p}$.\label{Fig0.1}}
\centering
   \end{figure}

Given a directed tree $\tcal=(V,E)$, let $\ell^2(V)$
be the Hilbert space of all square summable complex
functions on $V$ equipped with the usual inner
product. The family $\{e_u\}_{u\in V}$ defined by
   \begin{align*}
e_u(v) =
   \begin{cases}
1 & \text{if } v = u,
   \\
0 & \text{otherwise,}
   \end{cases}
\quad v\in V,
   \end{align*}
is clearly an orthonormal basis of $\ell^2(V)$. Given
a system $\lambdab = \{\lambda_v\}_{v\in V^{\circ}}
\subseteq \cbb$, we define the operator $\slam$ in
$\ell^2(V)$ by $\slam f = \varLambda_\tcal f$ for $f
\in \ell^2(V)$ such that $\varLambda_\tcal f \in
\ell^2(V)$, where $\varLambda_\tcal$ acts on functions
$f\colon V \rightarrow \cbb$ via
   \begin{align*}
(\varLambda_\tcal f)(v) =
   \begin{cases}
\lambda_v \cdot f(\mathrm{par}(v)) & \text{if $$} v
\in V^\circ,
   \\
0 & \text{if $v$ is a root of $\tcal$.}
   \end{cases}
   \end{align*}
We call $\slam$ the \textit{weighted shift} on the
directed tree $\tcal$ with weights $\{\lambda_v\}_{v
\in V^\circ}$. Throughout this paper it is assumed
that the resulting operator $\slam$ is bounded (see
\cite{EJSY1} or more generally \cite{JJS} for an
approach suitable even for unbounded shifts, and for
discussions of when $\slam$ is indeed bounded). If
$\slam \in \ogr{\ell^2(V)}$, in view of
\cite[(3.1.4)]{JJS}, one may more easily express
$\slam$ by
   \begin{align*}
\slam e_u = \sum_{v \in \dzii(u)} \lambda_v e_v.
   \end{align*}
(We adopt the convention that $\sum_{v\in \varnothing}
x_v = 0$.) The weighted shifts desired are those which
are subnormal operators. According to \cite[Theorem
6.1.3 and Notation 6.1.9]{JJS}, the following
assertion holds.
   \begin{align} \label{IlB}
   \begin{minipage}{68ex}
{\em $\slam \in \ogr{\ell^2(V)}$ is subnormal if and
only if for any $u\in V$, there exists $($a unique
compactly supported\/$)$ Borel measure on $\rbb_+$,
denoted by $\mu_u^{\lambdab}$, such that}
   $$ \|S_{\lambdab}^n e_u\|^2 = \int_{\rbb_+} t^n d
\mu_u^{\lambdab}(t), \quad n\in \zbb_+.
   $$
   \end{minipage}
   \end{align}
The characterizations of subnormality of $\slam$ on
the directed tree $\tcal_{\eta, \kappa}$ can be found
in \cite[Corollary~6.2.2]{JJS}.

Matters are in hand for the statement of the fundamental
problem considered. Let $\tcal=(V,E)$ be a subtree of a
directed tree $\hat{\tcal}=(\hat{V},\hat{E})$ and let
$\lambdab=\{\lambda_v\}_{v\in V^{\circ}}$ be a system of
positive real numbers. We say that a weighted shift
$\slamh$ on $\tcalh$ with weights
$\lambdabh=\{\hat{\lambda}_v\}_{v\in \hat{V}^{\circ}}
\subseteq (0,\infty)$ is a {\em subnormal completion} of
$\lambdab$ on $\tcalh$ if $\slamh\in
\ogr{\ell^2(\hat{V})}$, $\lambdab \subseteq \lambdabh$,
i.e., $\lambda_v=\hat{\lambda}_v$ for all $v\in V^{\circ}$,
and $\slamh$ is subnormal. If such a completion exists, we
may sometimes say $\lambdab$ \textit{admits a subnormal
completion} on $\tcalh$. The {\em subnormal completion
problem} for $(\tcal,\tcalh)$ consists of seeking necessary
and sufficient conditions for a system
$\lambdab=\{\lambda_v\}_{v\in V^{\circ}} \subseteq
(0,\infty)$ to have a subnormal completion on $\tcalh$.

Following \cite{EJSY1}, the subnormal completion
problem for $(\tcal_{\eta,\kappa,p},
\tcal_{\eta,\kappa})$, where $p\in \nbb$, is called
the {\em $p$-genera\-tion subnormal completion problem
on $\tcal_{\eta,\kappa}$} (sometimes abbreviated to
{\em $p$-generation SCP on $\tcal_{\eta,\kappa}$}). In
this particular case, our initial data takes the form
   \begin{align*}
\lambdab = \{\lambda_v\}_{v\in V_{\eta,
\kappa,p}^{\circ}} = \{\lambda_{-\kappa+1}, \ldots,
\lambda_0\} \cup \{\lambda_{i,1}\}_{i\in J_{\eta}}
\cup \ldots \cup \{\lambda_{i,p}\}_{i\in J_{\eta}}.
   \end{align*}

   The following result, which gives the measure-theoretic
way of solving the $p$-generation subnormal completion
problem on $\tcal_{\eta,\kappa}$, is a consequence of
\cite[Lemma~ 4.7 and Theorem~4.9]{EJSY1}.
   \begin{lem} \label{nproc}
Suppose $\eta \in \nbbc$, $\kappa, p \in \nbb$ and
$\lambdab = \{\lambda_v\}_{v\in V_{\eta,\kappa,p}^{\circ}}
\subseteq (0,\infty)$ are given. Then the following
conditions are equivalent{\em :}
   \begin{enumerate}
   \item[(i)] $\lambdab$ admits a subnormal completion
on $\tcal_{\eta,\kappa}$,
   \item[(ii)] there exist Borel
probability measures $\{\mu_{i}\}_{i=1}^{\eta}$ on $\rbb_+$
which satisfy the following conditions{\em :}
   \begin{align}  \label{q-mom1}
&\int_{\rbb_+} s^{n} d\mu_{i}(s) =
\prod_{j=2}^{n+1}\lambda_{i,j}^{2}, \quad n\in
J_{p-1},\ i\in J_{\eta},
   \\ \label{q-mom2}
&\sum_{i=1}^{\eta} \lambda_{i,1}^{2} \int_{\rbb_+}
\frac{1}{s} d \mu_{i}(s) =1,
   \\ \notag
&\sum_{i=1}^{\eta}\lambda_{i,1}^{2} \int_{\rbb_+}
\frac{1}{s^{k+1}} d\mu_{i}(s) =
\frac{1}{\prod_{j=0}^{k-1}\lambda_{-j}^{2}}, \quad
k\in J_{\kappa-1},
   \\ \label{q-mom3}
&\sum_{i=1}^{\eta} \lambda_{i,1}^{2} \int_{\rbb_+}
\frac{1}{s^{\kappa+1}} d \mu_{i}(s) \Le
\frac{1}{\prod_{j=0}^{\kappa-1} \lambda_{-j}^{2}},
   \\ \label{q-mom4}
&\sup_{i\in J_{\eta}} \sup \supp{\mu_{i}} < \infty.
   \end{align}
   \end{enumerate}
Moreover, if $\{\mu_{i}\}_{i=1}^{\eta}$ are as in {\em
(ii)}, then there exists a subnormal completion
$\slamh$ of $\lambdab$ on $\tcal_{\eta,\kappa}$ such
that $\mu_{i,1}^{\lambdabh}=\mu_{i}$ for all $i\in
J_{\eta}$.
   \end{lem}
In \cite[Theorem~5.1]{EJSY1} we gave an explicit
solution of the $1$-generation subnormal completion
problem on $\tcal_{\eta,1}$ written in terms of
initial data. As shown in \cite[Theorem~6.2]{EJSY1},
the $1$-generation subnormal completion problem on
$\tcal_{\eta,\kappa}$, where $\kappa\in \nbb$, reduces
to seeking necessary and sufficient conditions for a
system $\lambdab=\{\lambda_v\}_{v\in
V_{\eta,\kappa,1}^{\circ}} \subseteq (0,\infty)$ to
have a $2$-generation flat subnormal completion on
$\tcal_{\eta, \kappa}$. It was also proved in
\cite[Theorem~8.3]{EJSY1} that the problem of finding
a $p$-generation flat subnormal completion on
$\tcal_{\eta,\kappa}$ can be solved by using the
well-known solutions of the subnormal completion
problem for unilateral weighted shifts given in
\cite{Sta,CuF91,CuF93,CuF94,CuF1996,JJKS2011}.
However, in most cases these solutions are not written
explicitly in terms of initial data. This means that
we have no explicit (i.e., written in terms of initial
data) solution of the $p$-generation subnormal
completion problem on $\tcal_{\eta,\kappa}$ for $p\Ge
2$, even when $\kappa=1$. It is the right moment to
make the following important observation which is
implicitly contained in the proof of
\cite[Theorem~7.1]{EJSY1}.
   \begin{thm} \label{finite-su}
Suppose $\eta \in \nbb_2$, $\kappa, p \in \nbb$ and
$\lambdab = \{\lambda_v\}_{v\in V_{\eta,\kappa,p}^{\circ}}
\subseteq (0,\infty)$ are given. If $\lambdab$ admits a
subnormal completion on $\tcal_{\eta,\kappa}$, then it
admits a subnormal completion $\slamh$ on
$\tcal_{\eta,\kappa}$ such that
   \begin{align*}
\card{\supp{\mu_{i,1}^{\lambdabh}}} \Le
\Big\lfloor\frac{\kappa+p+2}{2}\Big\rfloor, \quad i\in
J_{\eta},
   \end{align*}
where $\lfloor x \rfloor = \min\big\{n\in \zbb_+\colon n
\Le x < n+1\big\}$ for $x\in \rbb_+$.
   \end{thm}
   \begin{proof}
Applying \cite[Theorems~5.1(iii) and 5.3(iii)]{CuF91}
and arguing as in the proof of the implication
(iii)$\Rightarrow$(iv) of \cite[Theorem~7.1]{EJSY1},
we may assume without loss of generality that for
every $i\in J_{\eta}$, the measure $\rho_i$ appearing
in the proof of the implication (iv)$\Rightarrow$(v)
of \cite[Theorem 7.1]{EJSY1} satisfies the following
condition:
   \begin{align*}
\text{$\card{\supp{\rho_i}} \Le
\Big\lfloor\frac{\kappa+p+2}{2}\Big\rfloor$ and
$\supp{\rho_i} \subseteq (0,\infty)$.}
   \end{align*}
Next, by arguing as in the proof of the implication
(iv)$\Rightarrow$(v) of \cite[Theorem 7.1]{EJSY1}, we
get a subnormal completion with the desired property.
   \end{proof}
It follows from Theorem~ \ref{finite-su} that if $\eta
\in \nbb_2$ and the $2$-generation subnormal
completion problem on $\tcal_{\eta,1}$ has a solution
for given data $\lambdab = \{\lambda_v\}_{v\in
V_{\eta,1,2}^{\circ}} \subseteq (0,\infty)$, then one
can always find a subnormal completion $\slamh$ of
$\lambdab$ on $\tcal_{\eta,1}$ with the property that
each measure $\mu_{i,1}^{\lambdabh}$ is $1$- or
$2$-atomic.
   \section{\label{Sect3}Preparatory lemmas} One of the goals of
this paper is to explore when initial data
$\lambdab=\{\lambda_v\}_{v\in V_{\eta,1,2}^{\circ}}$ on
$\tcal_{\eta,1,2}$ admits a subnormal completion $\slamh$
on $\tcal_{\eta,1}$ such that each of the measures
$\mu_{i,1}^{\lambdabh}$ (see \eqref{IlB}) is $2$-atomic
(under present circumstances, the measures
$\mu_{i,1}^{\lambdabh}$ may be taken to be $1$- or
$2$-atomic due to Theorem~\ref{finite-su}). We first
require some background on the Stampfli completion of three
increasing weights to the weight sequence for a subnormal
unilateral weighted shift (cf.\ \cite{Sta}).

In \cite{Sta} the author gives an explicit
construction of the completion of an initial finite
sequence of three weights $x,y,z$ satisfying $0 < x <
y < z$ to the sequence of positive weights
$\{\alpha_n\}_{n=0}^{\infty}$ for a (bounded,
injective) subnormal unilateral weighted shift
$W_{\alpha}$. This includes a construction of the
associated Berger measure of $W_{\alpha}$ (i.e., a
unique Borel probability measure $\xi$ on $\rbb_+$
such that for any $n\in \nbb$, $\gamma_n :=
\alpha_{0}^2 \cdots \alpha_{n-1}^2 = \int_{\rbb_+} t^n
d \xi(t)$), which turns out to be $2$-atomic. The
completion sequence $\{\alpha_n\}_{n=0}^{\infty}$,
which is called the {\em Stampfli completion of
$(x,y,z)$}, is customarily denoted by
$(x,y,z)^{\wedge}$; it is known that the moment
sequence $\{\gamma_n\}_{n=0}^{\infty}$ (with
$\gamma_{0}=1$) for $W_{\alpha}$ satisfies a
recursion.
   \begin{dfn}
The Berger measure of $W_{\alpha}$ will be called the
{\em Berger measure associated with the Stampfli
completion} $(x,y,z)^{\wedge}$ of $(x,y,z)$ and
denoted by $\xi_{x,y,z}$.
   \end{dfn}
One may see \cite{CuF91,CuF93,CuF94} for an
alternative approach to these same results. We note
also that one may consider the cases $0 < x < y = z$
(yielding a $2$-atomic measure with an atom at zero)
and $0 < x = y = z$, and this last case yields a
completion whose Berger measure is $1$-atomic.
%We will abuse language slightly so that to say a
%measure is $2$-atomic leaves open the possibility that
%it is $1$-atomic, and specify carefully when we insist
%that a measure be ``{\em genuinely}'' $2$-atomic.

Our technique will be to try to choose $y_i$ and
$z_i$, which will become respectively
$\hat\lambda_{i,3}$ and $\hat\lambda_{i,4}$ of the
completion $\slamh$, in such a way that the Berger
measures associated to $(x_{i},y_{i},z_{i})^{\wedge}$
will become the $\mu_{i,1}^{\lambdabh}$ and have good
properties (and of course they will automatically be
$2$-atomic). Figure~ \ref{Fig0-a} summarizes the task
and our notation.
\begin{figure}[ht]
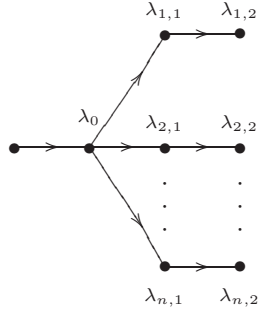

\centerline{ \xy (30,8);(36,8)**@{-} ?>*\dir{>},
(36,8);(40,8)**@{-}, (40,8);(47,18)**@{-} ?>*\dir{>},
(47,18);(50,22.8)**@{-},
(40,8);(46,8)**@{-}?>*\dir{>},(46,8);(50,8)**@{-},
(50,8);(56,8)**@{-}?>*\dir{>},(56,8);(60,8)**@{-},
(50,23);(56,23)**@{-}?>*\dir{>},(56,23);(60,23)**@{-},(40,8);(47,-2.66)**@{-}
?>*\dir{>}, (47,-2.66);(50,-8)**@{-},
(50,-8);(56,-8)**@{-}
?>*\dir{>},(56,-8);(60,-8)**@{-},
(40,12)*{_{\lambda_{0}}},(50,26)*{_{\lambda_{1,1}}},(50,11)*{_{\lambda_{2,1}}},
(50,-12)*{_{\lambda_{n,1}}},(50,3)*{\cdot},(50,0)*{\cdot},
(50,-3)*{\cdot},(60,3)*{\cdot},(60,0)*{\cdot},(60,-3)*{\cdot},
(60,26)*{_{\lambda_{1,2}}},(60,11)*{_{\lambda_{2,2}}},(60,-12)*{_{\lambda_{n,2}}},
(50,22.8)*{\bullet},(60,23)*{\bullet},
(30,7.7)*{\bullet},(40,7.7)*{\bullet},(50,7.7)*{\bullet},(60,7.7)*{\bullet},
(50,-8)*{\bullet},(60,-8)*{\bullet}
\endxy } \vspace*{5pt}
\caption{An illustration of a $2$-generation SCP on
$\tcal_{\eta,1}$ \newline with $\lambda_{0}$,
$\{\lambda_{i,1}\}_{i\in J_{\eta}}$ and
$\{\lambda_{i,2}\}_{i\in J_{\eta}}$ as given data.
\label{Fig0-a}} \centering
\end{figure}

We begin by calculating the first two ``negative''
moments of the Berger measure associated with the
Stampfli completion $(x,y,z)^{\wedge}$ of $(x,y,z)$.
   \begin{lem} \label{lem:StampfliInts}
Suppose that $(x,y,z)\in \rbb$ are such that $0 < x <
y < z$. Then
   \begin{align} \label{eq:int1ovsdmu}
\int_{\rbb_+} \frac{1}{s} d\xi_{x,y,z}(s) &=
\frac{x^{4}-2x^{2}y^{2}+
y^{2}z^{2}}{x^{2}y^{2}(z^{2}-y^{2})} = \frac{1}{x^{2}}
\left[\frac{1-2\frac{y^{2}}{x^{2}}
+\frac{y^{2}}{x^{2}}\frac{z^{2}}{x^{2}}}{\frac{y^{2}}{x^{2}}
(\frac{z^{2}}{x^{2}}-\frac{y^{2}}{x^{2}})} \right]
   \end{align}
and
   \begin{align} \label{eq:int1ovsSQdmu}
   \begin{aligned}
\int_{\rbb_+} \frac{1}{s^2}d\xi_{x,y,z}(s) & =
\frac{-x^{6}+y^{2}z^{4}+x^{4}(y^{2}+2z^{2})+x^{2}(y^{4}-4y^{2}z^{2})}
{x^{4}(y^{3}-yz^{2})^2}
   \\
& = \frac{1}{x^4} \left[\frac{-1+\frac{y^2}{x^2}
\cdot\frac{z^4}{x^4} +\frac{y^2}{x^2}
+\frac{2z^2}{x^2}+\frac{y^4}{x^4}
-\frac{4y^2}{x^2}\cdot\frac{z^2}{x^2}}
{\frac{y^2}{x^2} \big(\frac{z^2}{x^2} -
\frac{y^2}{x^2} \big)^2} \right].
   \end{aligned}
   \end{align}
   \end{lem}
   \begin{proof}
To simplify notation, set $\xi=\xi_{x,y,z}$. By
\cite[Example 3.14]{CuF93}, we have
   \begin{align*}
\xi= \rho \delta_{s_0} + (1-\rho) \delta_{s_1},
   \end{align*}
where $\rho=\frac{s_1-x^2}{s_1-s_0}$, $s_0=\frac 12(\psi_1
- \sqrt{\psi_1^2 + 4 \psi_0})$ and $s_1=\frac 12(\psi_1 +
\sqrt{\psi_1^2 + 4 \psi_0})$ with $\psi_0= - x^2y^2
\frac{z^2-y^2}{y^2-x^2}$ and $\psi_1= y^2
\frac{z^2-x^2}{y^2-x^2}$. Straightforward computations now
yield \eqref{eq:int1ovsdmu} and \eqref{eq:int1ovsSQdmu}.
   \end{proof}
   Next we investigate the function $f$ which comes from
the expression appearing on the right-hand side of the
second equality in \eqref{eq:int1ovsdmu}
   \begin{lem} \label{le:haty}
Let $f$ be a real function on $\varOmega=\{(u, v)\in
\rbb^2\colon v>u>1\}$ given~ by
   \begin{align*}
f(u,v) = \frac{1-2u+uv}{u (v-u)}, \quad (u, v)\in
\varOmega.
   \end{align*}
Then $f(\varOmega) = (1, \infty)$ and for any $(u,
v)\in \varOmega$ and $r\in(1,\infty)$,
   \begin{equation}  \label{eq:choicehatz}
f(u,v)=r \quad \text{if and only if} \quad v= \frac{1
- 2u + r u^2}{(r-1)u}.
   \end{equation}
Further, for any $u\in (1,\infty)$, the map
$\varphi_{u}\colon (1,\infty) \to (u, \infty)$ defined
by
   \begin{align} \label{fiju}
\varphi_{u}(r) =\frac{1 - 2u + r u^2} {(r-1)u}, \quad
r\in (1,\infty),
   \end{align}
is a bijection.
   \end{lem}
   \begin{proof}
It is a routine matter to verify that for any $u\in
(1,\infty)$, the function $\varphi_{u}$ is a
well-defined bijection from $(1,\infty)$ to $(u,
\infty)$. Clearly, the function $f$ is well defined.
Since $(u,\varphi_{u}(r)) \in \varOmega$ and
$f(u,\varphi_{u}(r))=r$ for all $r\in (1,\infty)$ and
$u\in (1,\infty)$, we deduce that $(1, \infty)
\subseteq f(\varOmega)$. Next, a simple argument shows
that $f(\varOmega) \subseteq (1,\infty)$, so
$f(\varOmega) = (1, \infty)$. It is a computation to
show that \eqref{eq:choicehatz} holds.
   \end{proof}
   Corollary \ref{sobot-a} below provides more information
on the behaviour of the first ``negative'' moment of the
Berger measure appearing in Lemma \ref{lem:StampfliInts}.
   \begin{cor} \label{sobot-a}
Let $x, y\in \rbb$ be such that $0 < x < y$. Then for
any $z\in (y,\infty)$, there exists a unique $r\in
(1,\infty)$ such that
   \begin{align}  \label{arov-1}
\int_{\rbb_+} \frac{1}{s} d\xi_{x,y,z}(s)=
r\frac{1}{x^2}.
   \end{align}
Conversely, for any $r \in (1,\infty)$, there exists a
unique $z\in (y,\infty)$ such that {\em
\eqref{arov-1}} holds; the number $z$ is determined by
the formula $v=\varphi_{u}(r)$ with $u =
\frac{y^2}{x^2}$ and $v = \frac{z^2}{x^2}$.
   \end{cor}
   \begin{proof}
Making the substitutions $u = \frac{y^2}{x^2}$ and $v
= \frac{z^2}{x^2}$, we see that $(u,v)\in \varOmega$
and
   \begin{equation*}
\frac{1}{x^{2}} \left[\frac{1-2\frac{y^{2}}{x^{2}} +
\frac{y^{2}} {x^{2}} \frac{z^{2}}{x^{2}}}
{\frac{y^{2}}{x^{2}}(\frac{z^{2}}{x^{2}}-\frac{y^{2}}{x^{2}})}
\right] = \frac{1}{x^{2}} f(u,v).
   \end{equation*}
Now applying Lemmata \ref{lem:StampfliInts} and
\ref{le:haty} completes the proof.
   \end{proof}
   The expression on the right-hand side of the second
equality in \eqref{eq:int1ovsSQdmu} leads to the
function $g$ which appears in a lemma below. The proof
of this lemma follows from straightforward
computations via Lemma \ref{le:haty}. The details are
left to the reader.
   \begin{lem}  \label{to-to}
Let $\varOmega$ be as in Lemma~ {\em \ref{le:haty}}
and $g$ be a real function on $\varOmega$ given~ by
   \begin{equation*}
g(u,v) = \frac{-1+u \cdot v^2 + u + 2v + u^2 - 4u
\cdot v} {u(v-u)^2}, \quad (u,v)\in \varOmega.
   \end{equation*}
For $r \in (1,\infty)$, let $h_r$ be a real function
on $(1, \infty)$ defined by\/\footnote{\;In view of
Lemma \ref{le:haty}, the definition of $h_{r}$ is
correct.}
   \begin{align*}
h_r(u) = g\big(u, \varphi_u(r)\big), \quad u \in
(1,\infty),
   \end{align*}
where $\varphi_u$ is as in \eqref{fiju}. Then the
following statements hold for each $r\in (1,\infty)$:
   \begin{enumerate}
  \item[(i)] $h_r(u) = \frac{1-2r + r^2 u}{u - 1},  \quad u\in (1,\infty)$,
  \item[(ii)] $\lim_{u \to 1 +}h_r(u) = \infty$,
  \item[(iii)] $\lim_{u \to \infty} h_r(u) = r^2$,
  \item[(iv)] $h_r\big((1,\infty)\big) = (r^2, \infty)$ and
$h_r\colon (1,\infty) \to (r^2, \infty)$ is a bijection.
   \end{enumerate}
%Moreover, for any $(u,v)\in \varOmega$, there exist
%$r, \vartheta \in (1,\infty)$ such that
%   \begin{align} \label{Rzymm}
%g(u,v) = \vartheta \, r^2.
%   \end{align}
   \end{lem}
%   \begin{proof} Straightforward computations yield
%the statements (i)-(iv). To prove the ``moreover''
%part, take $(u,v)\in \varOmega$. Since $v\in
%(u,\infty)$, we infer from Lemma~ \ref{le:haty} that
%there exists $r\in (1,\infty)$ such that
%$v=\varphi_{u}(r)$, and thus by (iv) we have
%   \begin{align*}
%g(u,v)=h_r(u) \in (r^2,\infty),
%   \end{align*}
%which implies \eqref{Rzymm}.
%   \end{proof}
Putting together the last three lemmas, we obtain the
following crucial lemma.
   \begin{lem}  \label{prop:formsoftheints}
For any $(x,y,z)\in \rbb^3$ such that $0 < x < y < z$,
there exist $r, \vartheta \in (1,\infty)$ such that
   \begin{align} \label{Rzym-1}
\int_{\rbb_+} \frac{1}{s}d\xi_{x,y,z}(s) & = r
\frac{1}{x^2},
   \\ \label{Rzym-2}
\int_{\rbb_+} \frac{1}{s^2}d\xi_{x,y,z}(s)& =
\vartheta \, r^2 \frac{1}{x^4}.
   \end{align}
Moreover, for any $x\in (0,\infty)$ and any $r, \vartheta
\in (1,\infty)$, there exists $(y,z)\in \rbb^2$ with $x < y
< z$ satisfying \eqref{Rzym-1} and \eqref{Rzym-2}.
   \end{lem}
   \begin{proof}
Assume $(x,y,z)\in \rbb^3$ is such that $0 < x < y <
z$. Set $u=\frac{y^2}{x^2}$ and $v=\frac{z^2}{x^2}$.
Then $(u,v)\in \varOmega$. By Lemma \ref{le:haty},
$r:=f(u,v) \in (1,\infty)$ and $v=\varphi_u(r)$, so
   \begin{align*}
\int_{\rbb_+} \frac{1}{s}
d\xi_{x,y,z}(s)\overset{\eqref{eq:int1ovsdmu}}=\frac{1}{x^2}
f(u,v) = r \frac{1}{x^2},
   \end{align*}
which gives \eqref{Rzym-1}. In turn, by Lemma~
\ref{to-to}(iv), $h_r(u) \in (r^2,\infty)$ and
   \begin{align*}
\int_{\rbb_+} \frac{1}{s^2}d\xi_{x,y,z}(s)
\overset{\eqref{eq:int1ovsSQdmu}} = \frac{1}{x^4}
g(u,v) = \frac{1}{x^4} h_r(u),
   \end{align*}
which implies that \eqref{Rzym-2} holds for some
$\vartheta \in (1,\infty)$.

Suppose now that $x\in (0,\infty)$ and $r, \vartheta
\in (1,\infty)$. Since $\vartheta \, r^2 \in
(r^2,\infty)$, we infer from Lemma \ref{to-to}(iv)
that there exists $u\in (1,\infty)$ such that
$h_r(u)=\vartheta \, r^2$. Set $v=\varphi_u(r)$. Then
$(u, v) \in \varOmega$ and
   \begin{align} \label{niedz-1}
g\big(u, v\big) = h_r(u)=\vartheta \, r^2.
   \end{align}
Setting $y=x \sqrt{u}$ and $z=x \sqrt{v}$, we see that
$0<x < y < z$ and
   \begin{align*}
\int_{\rbb_+} \frac{1}{s^2}d\xi_{x,y,z}(s)
\overset{\eqref{eq:int1ovsSQdmu}}=\frac{1}{x^4}
g\big(u, v\big) \overset{ \eqref{niedz-1}}= \vartheta
\, r^2 \frac{1}{x^4},
   \end{align*}
which gives \eqref{Rzym-2}. Since $v=\varphi_u(r)$, we
have
   \begin{align*}
\int_{\rbb_+} \frac{1}{s} d\xi_{x,y,z}(s)
\overset{\eqref{eq:int1ovsdmu}}= \frac{1}{x^2} f(u,v)
= \frac{1}{x^2} f(u,\varphi_u(r))
\overset{\eqref{eq:choicehatz}} = r \frac{1}{x^2},
   \end{align*}
which implies \eqref{Rzym-1}. This completes the
proof.
   \end{proof}
   \section{\label{Sect4}The \mbox{$2$}-generation SCP on
\mbox{$\tcal_{\eta,1}$}}
   We begin by solving the $2$-generation subnormal
completion problem on $\tcal_{\eta,1}$ with $2$-atomic
measures (the case $\eta=\infty$ is included).
   \begin{thm} \label{charoftwoatomic-infty}
Suppose $\eta \in \nbbc$, $\kappa = 1$, $p=2$ and $\lambdab
= \{\lambda_0\} \cup \{\lambda_{i,1}\}_{i=1}^{\eta} \cup
\{\lambda_{i,2}\}_{i=1}^{\eta} \subseteq (0,\infty)$ are
given. Then the following statements are equivalent{\em :}
   \begin{enumerate}
   \item[(i)] $\lambdab$ has a subnormal completion $\slamh$ on
$\tcal_{\eta, 1}$ such that each measure $\mu_{i,
1}^{\lambdabh}$ is $2$-atomic,
   \item[(ii)] there exist sequences $\{r_i\}_{i=1}^{\eta},
\{\vartheta_i\}_{i=1}^{\eta} \subseteq (1,\infty)$ such~
that
   \begin{gather} \label{sum-00-n-infty}
\sup_{i\in J_{\eta}} \frac{\vartheta_i r_i -1}{(\vartheta_i
- 1)r_i} \lambda_{i,2}^2 < \infty,
   \\ \label{sum-0-n-infty}
\sum_{i=1}^{\eta} r_i \frac{\lambda_{i,1}^2}
{\lambda_{i,2}^2} = 1,
   \\ \label{sum-01-n-infty}
\sum_{i=1}^{\eta} \vartheta_i r_i^2 \frac{\lambda_{i,1}^2}
{\lambda_{i,2}^4} \Le \frac{1}{\lambda_0^2}.
   \end{gather}
   \end{enumerate}
Moreover, if {\em (i)} holds, then $\mu_{i,
1}^{\lambdabh}(\{0\})=0$ for all $i\in J_{\eta}$~and
   \begin{gather} \label{kiki-1-infty}
\lambda_0^2 < \sum_{i=1}^{\eta} \lambda_{i,1}^2, \quad
\sum_{i=1}^{\eta} \frac{\lambda_{i,1}^2} {\lambda_{i,2}^4}
< \frac{1}{\lambda_0^2}, \quad \sum_{i=1}^{\eta}
\frac{\lambda_{i,1}^2} {\lambda_{i,2}^2} < 1, \quad
\sum_{i=1}^{\eta} \lambda_{i,1}^2 < \sup_{i\in J_{\eta}}
\lambda_{i,2}^2.
   \end{gather}
   \end{thm}
   \begin{proof}
We concentrate on proving the case when $\eta=\infty$.
If $\eta < \infty$, the proof simplifies. In
particular, the statement \eqref{sum-00-n-infty} can
be dropped.

(i)$\Rightarrow$(ii) Let $\slamh$ be a subnormal completion
of $\lambdab$ on $\tcal_{\infty, 1}$ such that each measure
$\mu_i:=\mu_{i, 1}^{\lambdabh}$ is $2$-atomic. It follows
from \cite[Corollary~6.2.2(ii)]{JJS} that
   \allowdisplaybreaks
   \begin{align} \label{sum-1-n-infty}
\sum_{i=1}^{\infty }\lambda _{i,1}^{2}\int_{\rbb_+}
\frac{1}{s}d\mu_{i}(s) & =1,
   \\  \label{sum-2-n-infty}
\sum_{i=1}^{\infty}\lambda _{i,1}^{2}\int_{\rbb_+}
\frac{1}{s^2} d\mu_{i}(s) & \Le \frac{1}{\lambda_0^{2}}.
   \end{align}
Applying \eqref{IlB} to $u=e_{i,1}$ and using the
Berger-Gellar-Wallen theorem (see \cite{GW,Hal}), we deduce
that for every $i\in \nbb$, the unilateral weighted shift
$W_{\alpha^{(i)}}$ with weights $\alpha^{(i)}=\{\hat
\lambda_{i,n+2}\}_{n=0}^{\infty}$ is subnormal and $\mu_i$
is the Berger measure of $W_{\alpha^{(i)}}$. It follows
from \eqref{sum-1-n-infty} that $\mu_i$ cannot have an atom
at zero. Hence each $\mu_i$ has two atoms in $(0,\infty)$.
As a consequence, $\lambda_{i,2} < \hat\lambda_{i,3} <
\hat\lambda_{i,4}$ and
$\mu_i=\xi_{\lambda_{i,2},\hat\lambda_{i,3},
\hat\lambda_{i,4}}$ for every $i \in \nbb$ (see
\cite[Lemma~2.3]{EJSY1} and \cite[Example 3.14, Theorem
3.9(iii)]{CuF93}, respectively). Applying Lemma~
\ref{prop:formsoftheints} to $x=x_i:=\lambda_{i,2}$,
$y=y_i:=\hat\lambda_{i,3}$ and $z=z_i:=\hat\lambda_{i,4}$,
we deduce that there exist sequences
$\{r_i\}_{i=1}^{\infty} \subseteq (1,\infty)$ and
$\{\vartheta_i\}_{i=1}^{\infty} \subseteq (1,\infty)$ such
that \allowdisplaybreaks
   \begin{align} \label{sum-3-n-infty}
\int_{\rbb_+} \frac{1}{s}d\mu_{i}(s) & = r_i
\frac{1}{\lambda_{i,2}^2}, \quad i \in \nbb,
   \\ \label{sum-4-n-infty}
\int_{\rbb_+} \frac{1}{s^2}d\mu_i(s)& = \vartheta_i \,
r_i^2 \frac{1}{\lambda_{i,2}^4}, \quad i \in \nbb.
   \end{align}
According to the proof of Lemma~\ref{prop:formsoftheints},
the sequences $\{r_i\}_{i=1}^{\infty}$ and
$\{\vartheta_i\}_{i=1}^{\infty}$ are constructed via the
following process:
   \begin{gather}  \label{circum-1}
   \begin{gathered}
\text{$y_i=x_i\sqrt{u_i}$, $z_i=x_i\sqrt{v_i}$,
$v_i=\varphi_{u_i}(r_i)$ and $h_{r_i}(u_i)=\vartheta_i
r_i^2$ for $i\in \nbb$,}
   \\
\text{where $0<x_i<y_i<z_i$ and $1 < u_i < v_i$.}
   \end{gathered}
   \end{gather}
(Recall that by Lemmata~\ref{le:haty} and \ref{to-to}, the
functions $\varphi_{u_i}\colon (1,\infty) \to (u_i,
\infty)$ and $h_{r_i}\colon (1,\infty) \to (r_i^2, \infty)$
are bijections.) Combining \eqref{sum-1-n-infty} with
\eqref{sum-3-n-infty}, we obtain \eqref{sum-0-n-infty}. In
turn, \eqref{sum-2-n-infty} and \eqref{sum-4-n-infty} imply
\eqref{sum-01-n-infty}. To get (ii), it remains to prove
\eqref{sum-00-n-infty}.

For this, we show that under the circumstances of
\eqref{circum-1} the following assertion holds:
   \begin{align} \label{com-4}
\varXi:=\sup_{i\in \nbb}\sup\supp \xi_{x_i,y_i,z_i} <
\infty \iff \sup_{i\in \nbb} x_i^2\frac{\vartheta_i r_i
-1}{(\vartheta_i-1)r_i} < \infty.
   \end{align}
Indeed, according to \cite[Example 3.14]{CuF93}, we have
   \begin{align} \label{Vangel1}
\sup\supp \xi_{x_i,y_i,z_i} = \frac{1}{2} \bigg(\psi_{1,i}
+ \sqrt{\psi_{1,i}^2 + 4\psi_{0,i}}\bigg), \quad i\in \nbb,
   \end{align}
where $\psi_{0,i}= - x_i^2y_i^2
\frac{z_i^2-y_i^2}{y_i^2-x_i^2}$ for $i\in \nbb$ and
   \begin{align} \label{comb-0}
\psi_{1,i}= y_i^2 \frac{z_i^2-x_i^2}{y_i^2-x_i^2}
\overset{\eqref{circum-1}}= x_i^2 u_i \frac{v_i-1}{u_i-1},
\quad i\in \nbb.
   \end{align}
Since $\psi_{0,i} < 0$ and $\psi_{1,i}^2 + 4\psi_{0,i}>0$
for every $i\in \nbb$, we deduce from \eqref{Vangel1} that
   \begin{align} \label{comb-3}
\varXi < \infty \iff \sup_{i\in \nbb} \psi_{1,i} < \infty.
   \end{align}
By \eqref{circum-1}, $h_{r_i}(u_i)=\vartheta_i r_i^2$ for
all $i\in \nbb$, so Lemma~ \ref{to-to}(i) leads to
   \begin{align} \label{comb-1}
u_i = \frac{1-2r_i+\vartheta_i
r_i^2}{(\vartheta_i-1)r_i^2}, \quad i\in \nbb.
   \end{align}
Using the identity $v_i=\varphi_{u_i}(r_i)$ and
\eqref{fiju}, we get
   \begin{align} \label{comb-2}
v_i-1 = \frac{(u_i-1)(r_i u_i -1)}{(r_i-1)u_i}, \quad i \in
\nbb.
   \end{align}
Combining \eqref{comb-0}, \eqref{comb-1} and
\eqref{comb-2}, we see that
   \begin{align*}
\psi_{1,i} = \frac{x_i^2}{r_i-1} (r_iu_i-1) =
x_i^2\frac{\vartheta_i r_i -1}{(\vartheta_i-1)r_i}, \quad
i\in \nbb.
   \end{align*}
This together with \eqref{comb-3} yields \eqref{com-4}. It
follows from \cite[Eg.\ (6.3.9)]{JJS} that $\varXi <
\infty$. Hence, by \eqref{com-4}, \eqref{sum-00-n-infty} is
satisfied, which shows that (ii) holds.

(ii)$\Rightarrow$(i) Suppose now that (ii) holds.
Applying the ``moreover'' part of Lemma~
\ref{prop:formsoftheints} to $x=x_i:=\lambda_{i,2}$,
$r=r_i$ and $\vartheta=\vartheta_i$, we deduce that
for every $i\in \nbb$, there exists $(y_i,z_i) \in
\rbb^2$ with $x_i < y_i < z_i$ such that
\eqref{sum-3-n-infty} and \eqref{sum-4-n-infty} hold
with $\mu_i= \xi_{x_i, y_i, z_i}$. According to the
proof of Lemma~\ref{prop:formsoftheints}, the
sequences $\{y_i\}_{i=1}^{\infty}$ and
$\{z_i\}_{i=1}^{\infty}$ are constructed via the
procedure \eqref{circum-1}. It follows from
\eqref{sum-00-n-infty} and \eqref{com-4} that
   \begin{align} \label{mom-1-f-infty}
\sup_{i\in \nbb}\sup\supp \mu_i < \infty.
   \end{align}
Since in general $\int_{\rbb_+} s d\xi_{x,y,z}(s) = x^2$
whenever $0 < x < y < z$, we get
   \begin{align} \label{mom-1-c-infty}
\int_{\rbb_+} s d\mu_i(s) = \lambda_{i,2}^2, \quad i \in
\nbb.
   \end{align}
It follows that \allowdisplaybreaks
   \begin{gather} \label{mom-1-d-infty}
\sum_{i=1}^{\infty} \lambda_{i,1}^{2} \int_{\rbb_+}
\frac{1}{s} d\mu_{i}(s) \overset{\eqref{sum-3-n-infty}}=
\sum_{i=1}^{\infty} r_i \frac{\lambda_{i,1}^{2}}
{\lambda_{i,2}^2} \overset{\eqref{sum-0-n-infty}}= 1,
   \\ \label{mom-1-e-infty}
\sum_{i=1}^{\infty} \lambda_{i,1}^{2} \int_{\rbb_+}
\frac{1}{s^{2}} d\mu_{i}(s)
\overset{\eqref{sum-4-n-infty}}= \sum_{i=1}^{\infty}
\vartheta _i r_i^2 \frac{\lambda_{i,1}^{2}}
{\lambda_{i,2}^4} \overset{\eqref{sum-01-n-infty}} \Le
\frac{1}{\lambda_{0}^{2}}.
   \end{gather}
Putting together the conditions \eqref{mom-1-f-infty},
\eqref{mom-1-c-infty}, \eqref{mom-1-d-infty} and
\eqref{mom-1-e-infty}, and applying Lemma~ \ref{nproc}, we
conclude that $\lambdab$ has a subnormal completion
$\slamh$ on $\tcal_{\infty,1}$ such that $\mu_{i,
1}^{\lambdabh} = \mu_i=\xi_{x_i, y_i, z_i}$ for every $i\in
\nbb$. As a consequence, each $\mu_{i, 1}^{\lambdabh}$ is
$2$-atomic, which gives (i).

Now we prove the ``moreover'' part. Suppose (i) is
satisfied. The fact that $\mu_{i,
1}^{\lambdabh}(\{0\})=0$ for every $i\in \nbb$ is a
direct consequence of \cite[Theorem~ 3.5(i)]{EJSY1}.
Notice that the second inequality in
\eqref{kiki-1-infty} follows from
\eqref{sum-01-n-infty} while the third and the fourth
can be deduced from \eqref{sum-0-n-infty}. Thus, it
remains to prove the first inequality in
\eqref{kiki-1-infty}. It follows from \cite[Theorem~
4.9(i)]{EJSY1} that $\lambda_0^2\Le
\sum_{i=1}^{\infty} \lambda_{i,1}^2$. Suppose, to the
contrary, that $\lambda_0^2 = \sum_{i=1}^{\infty}
\lambda_{i,1}^2$. Then, by the ``moreover'' part of
\cite[Propositon~ 7.6]{EJSY1}, we see that
$\mu_{i,1}^{\lambdabh} = \delta_{c}$ for every $i\in
\nbb$ with $c=\sum_{i=1}^{\infty} \lambda_{i,1}^2$,
which contradicts our assumption that each
$\mu_{i,1}^{\lambdabh}$ is $2$-atomic. This completes
the proof.
   \end{proof}
   Now we are ready to solve the $2$-generation subnormal
completion problem on $\tcal_{\eta,1}$ for $\eta\in\nbb_2$
without any restrictions on the supports of the resulting
measures~$\mu_{i,1}^{\lambdabh}$.
   \begin{thm} \label{charoftwoatomic-solution}
Suppose $\eta \in \nbb_2$, $\kappa = 1$, $p=2$ and
$\lambdab = \{\lambda_0\} \cup
\{\lambda_{i,1}\}_{i=1}^{\eta} \cup
\{\lambda_{i,2}\}_{i=1}^{\eta} \subseteq (0,\infty)$ are
given. Then the following statements are equivalent{\em :}
   \begin{enumerate}
   \item[(i)] $\lambdab$ has a subnormal completion $\slamh$ on
$\tcal_{\eta, 1}$,
   \item[(ii)] there exists
a sequence $\{r_i\}_{i=1}^{\eta} \subseteq [1,\infty)$ such
that
   \begin{gather} \label{sum-0-n-sk}
\sum_{i=1}^{\eta} r_i \frac{\lambda_{i,1}^2}
{\lambda_{i,2}^2} = 1,
   \\ \label{sum-01-n-sk}
\sum_{i=1}^{\eta} r_i^2 \frac{\lambda_{i,1}^2}
{\lambda_{i,2}^4} \;
   \begin{cases}
\Le \frac{1}{\lambda_0^2} & \text{if $r_i=1$ for every
$i\in J_{\eta}$,}
   \\[2ex]
< \frac{1}{\lambda_0^2} & \text{\text{if $r_i > 1$ for
some $i\in J_{\eta}$}.}
   \end{cases}
   \end{gather}
   \end{enumerate}
Moreover, if $\slamh$ is a subnormal completion of
$\lambdab$ on $\tcal_{\eta, 1}$ such that each measure
$\mu_{i,1}^{\lambdabh}$ is $1$- or $2$-atomic, then
\eqref{sum-0-n-sk} and \eqref{sum-01-n-sk} hold for
some $\{r_i\}_{i=1}^{\eta} \subseteq [1,\infty)$
such~that
   \begin{align} \label{one-two-1}
\text{$r_i=1$ $($resp.\ $r_i> 1$$)$ if and only if
$\mu_{i, 1}^{\lambdabh}$ is $1$-atomic $($resp.\
$2$-atomic$)$.}
   \end{align}
Conversely, if {\em (ii)} holds, then $\lambdab$
admits a subnormal completion $\slamh$ on
$\tcal_{\eta,1}$ which satisfies \eqref{one-two-1}.
   \end{thm}
   \begin{proof}
In view of Theorem~ \ref{finite-su}, the statement (i)
is equivalent to the fact that $\lambdab$ admits a
subnormal completion $\slamh$ on $\tcal_{\eta,1}$ such
that each measure $\mu_{i, 1}^{\lambdabh}$ is $1$- or
$2$-atomic. Hence, in the proof of the implication
(i)$\Rightarrow$(ii), we can define the partition
$(A,B)$ of $J_{\eta}$ as follows
   \begin{align*}
\text{$A=\big\{i \in J_{\eta}\colon \mu_{i, 1}^{\lambdabh}
\text{ is $1$-atomic}\big\}$ and $B=\big\{i \in
J_{\eta}\colon \mu_{i, 1}^{\lambdabh} \text{ is
$2$-atomic}\big\}$.}
   \end{align*}
By \cite[Corollary~6.2.2(ii)]{JJS} the measures
$\mu_i=\mu_{i, 1}^{\lambdabh}$, $i\in J_{\eta}$,
satisfy \eqref{q-mom1}-\eqref{q-mom3} with $\kappa=1$
and $p=2$. If $B=\varnothing$, then it is easily seen
that $\mu_{i, 1}^{\lambdabh} =
\delta_{\lambda_{i,2}^2}$ for every $i\in J_{\eta}$
and so \eqref{sum-0-n-sk} and \eqref{sum-01-n-sk} hold
with $r_i=1$ for all $i\in J_{\eta}$. If $B\neq
\varnothing$, then by combining reasonings used above
and in the proof of the implication
(i)$\Rightarrow$(ii) of
Theorem~\ref{charoftwoatomic-infty}, we get two
sequences $\{r_i\}_{i\in B}, \{\vartheta_i\}_{i\in B}
\subseteq (1,\infty)$ such that
   \begin{gather} \label{dud-1}
\sum_{i\in A} \frac{\lambda_{i,1}^2} {\lambda_{i,2}^2} +
\sum_{i\in B} r_i \frac{\lambda_{i,1}^2} {\lambda_{i,2}^2}
= 1,
   \\ \label{dud-2}
\sum_{i\in A} \frac{\lambda_{i,1}^2}
{\lambda_{i,2}^4} + \sum_{i\in B} \vartheta_i r_i^2
\frac{\lambda_{i,1}^2} {\lambda_{i,2}^4} \Le
\frac{1}{\lambda_0^2}.
   \end{gather}
Since $\vartheta_i > 1$ for any $i\in B$, we see that
\eqref{dud-1} and \eqref{dud-2} imply \eqref{sum-0-n-sk}
and \eqref{sum-01-n-sk} with $r_i=1$ for $i\in A$. Putting
all of this together yields (ii) and \eqref{one-two-1}.

To prove the converse implication (ii)$\Rightarrow$(i), we
will define a family $\{\mu_i\}_{i\in J_{\eta}}$ of Borel
probability measures on $\rbb_+$ satisfying
\eqref{q-mom1}-\eqref{q-mom4} with $\kappa=1$ and $p=2$.
First, we define the partition $(A',B')$ of $J_{\eta}$ by
   \begin{align*}
\text{$A'=\{i\in J_{\eta}\colon r_i=1\}$ and
$B'=\{i\in J_{\eta}\colon r_i> 1\}$.}
   \end{align*}
If $B' = \varnothing$, then the probability measures
$\mu_i=\delta_{\lambda_{i,2}^2}$, $i\in J_{\eta}$, satisfy
\eqref{q-mom1}-\eqref{q-mom4}. If $B' \neq \varnothing$,
then we proceed as follows. By \eqref{sum-01-n-sk} there
exists $\tau \in (1,\infty)$ such that
   \begin{align}  \label{abdwa}
\sum_{i\in A'} \frac{\lambda_{i,1}^2}
{\lambda_{i,2}^4} + \tau \sum_{i\in B'} r_i^2
\frac{\lambda_{i,1}^2} {\lambda_{i,2}^4} <
\frac{1}{\lambda_0^2}.
   \end{align}
If $i \in A'$, then we set
$\mu_i=\delta_{\lambda_{i,2}^2}$. If $i\in B'$, then
we define $\mu_i$ as in the proof of the implication
(ii)$\Rightarrow$(i) of
Theorem~\ref{charoftwoatomic-infty} with
$\vartheta_i=\tau$. As in that proof, we verify that
   \allowdisplaybreaks
   \begin{align*}
\int_{\rbb_+} s d\mu_i(s) & = \lambda_{i,2}^2, \quad i \in
J_{\eta},
   \\
\sum_{i=1}^{\eta} \lambda_{i,1}^{2} \int_{\rbb_+}
\frac{1}{s} d\mu_{i}(s) & = \sum_{i\in A'}
\frac{\lambda_{i,1}^2} {\lambda_{i,2}^2} + \sum_{i\in
B'} r_i \frac{\lambda_{i,1}^2} {\lambda_{i,2}^2}
\overset{\eqref{sum-0-n-sk}}= 1,
   \\
\sum_{i=1}^{\eta} \lambda_{i,1}^{2} \int_{\rbb_+}
\frac{1}{s^{2}} d\mu_{i}(s) & = \sum_{i\in A'}
\frac{\lambda_{i,1}^2} {\lambda_{i,2}^4} + \sum_{i\in
B'} \vartheta_i r_i^2 \frac{\lambda_{i,1}^{2}}
{\lambda_{i,2}^4} \overset{\eqref{abdwa}} <
\frac{1}{\lambda_{0}^{2}}.
   \end{align*}
Hence \eqref{q-mom1}-\eqref{q-mom4} hold. Now, by applying
Lemma~ \ref{nproc}, we conclude that $\lambdab$ has a
subnormal completion $\slamh$ on $\tcal_{\eta,1}$ such that
$\mu_{i, 1}^{\lambdabh} = \mu_i$ for every $i\in J_{\eta}$,
which gives (i). Clearly, the measures $\{\mu_{i,
1}^{\lambdabh}\colon i\in A'\}$ are $1$-atomic while the
measures $\{\mu_{i, 1}^{\lambdabh}\colon i \in B'\}$ are
$2$-atomic. This completes the proof of the ``moreover''
part and the proof of the theorem.
   \end{proof}
It follows from Theorem~\ref{charoftwoatomic-solution} that
the $2$-generation subnormal completion problem on
$\tcal_{\eta,1}$ may have a solution admitting
simultaneously $1$- and $2$-atomic measures. Below we give
an example showing that the problem on $\tcal_{2,1}$ may
have a solution of this kind with distinct atoms.
   \begin{figure}[ht]
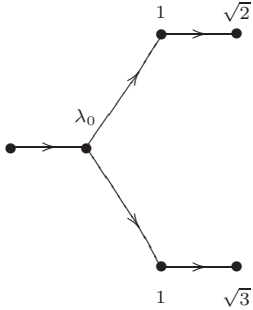

\centerline{ \xy (30,8);(36,8)**@{-} ?>*\dir{>},
(36,8);(40,8)**@{-}, (40,8);(47,18)**@{-} ?>*\dir{>},
(47,18);(50,22.8)**@{-}, (40,8);,
(50,23);(56,23)**@{-}?>*\dir{>},(56,23);(60,23)**@{-},(40,8);(47,-2.66)**@{-}
?>*\dir{>}, (47,-2.66);(50,-8)**@{-}, (50,-8);(56,-8)**@{-}
?>*\dir{>},(56,-8);(60,-8)**@{-},
(40,12)*{_{\lambda_0}},(50,26)*{_{1}},(50,-12)*{_{1}},
(60,26)*{_{\sqrt{2}}},(60,-12)*{_{\sqrt{3}}},(80,20)*{_{
\text{}}},(83,17)*{_{}},(80,12)*{_{ \text{}}},(80,7)*{_{
\text{}}},(80,2)*{_{ \text{}}},
(50,22.8)*{\bullet},(60,23)*{\bullet},
(30,7.7)*{\bullet},(40,7.7)*{\bullet},
(50,-8)*{\bullet},(60,-8)*{\bullet}
\endxy } \vspace*{5pt}
\caption{Mixed completions discussed in
Example~\ref{mix-1}. \label{Fig4}} \centering
\end{figure}
   \begin{exa} \label{mix-1}
Consider the $2$-generation subnormal completion problem on
$\tcal_{2,1}$ with the initial data $\lambdab =
\{\lambda_0\} \cup \{\lambda_{i,1}\}_{i=1}^{2} \cup
\{\lambda_{i,2}\}_{i=1}^{2} \subseteq (0,\infty)$, where
$\lambda_0 \in (0,\infty)$ is arbitrary and the remaining
weights are given by (see Figure \ref{Fig4})
   \begin{align*}
\text{$\lambda_{1,1}=\lambda_{2,1}=1$,
$\lambda_{1,2}=\sqrt{2}$ and $\lambda_{2,2}=\sqrt{3}$.}
   \end{align*}
We are looking for $\lambda_0$ for which $\lambdab$ admits
a subnormal completion $\slamh$ on $\tcal_{2,1}$ such that
$\mu_{1,1}^{\lambdabh}$ is $1$-atomic while
$\mu_{2,1}^{\lambdabh}$ is $2$-atomic, and all these atoms
are different. In view of \cite[Theorem~ 4.9]{EJSY1}, any
subnormal completion $\slamh$ of the above $\lambdab$ on
$\tcal_{2,1}$ comes from compactly supported Borel
probability measures $\mu_1$ and $\mu_2$ on $\rbb_+$ which
satisfy the following conditions
   \begin{gather} \label{dodo1}
\int_{\rbb_+} s d\mu_{i}(s) \overset{}= \lambda_{i,2}^2 =
   \begin{cases}
2 & \text{if } i=1,
   \\
3 & \text{if } i=2,
   \end{cases}
   \\ \label{dodo2}
\int_{\rbb_+} \frac{1}{s}d\mu_{1}(s) + \int_{\rbb_+}
\frac{1}{s}d\mu_{2}(s) =1,
   \\ \label{dodo3}
\int_{\rbb_+} \frac{1}{s^{2}} d\mu_{1}(s) + \int_{\rbb_+}
\frac{1}{s^{2}} d\mu_{2}(s) \Le \frac{1}{\lambda_{0}^{2}}.
   \end{gather}
Let $\mu_1$ and $\mu_2$ be Borel probability measures on
$\rbb_+$ given by
   \begin{align*}
\mu_1=\delta_x \quad \text{and} \quad \mu_2 = \frac 12
\delta_y + \frac 12 \delta_z.
   \end{align*}
It is a matter of routine to verify that $\mu_1$ and
$\mu_2$ with $x=2$, $y=3-\sqrt{3}$ and $z=3+\sqrt{3}$
satisfy \eqref{dodo1} and \eqref{dodo2}. Hence
\eqref{dodo3} holds if and only if $\lambda_0^2 \Le
\frac{12}{7}$, meaning that for such $\lambda_0$'s the
corresponding $\lambdab$ admits a subnormal completion on
$\tcal_{2,1}$ with the desired properties.
   \end{exa}
If $\eta \in \nbb_2$, Theorem~\ref{charoftwoatomic-infty}
takes a much simpler form which is a direct consequence of
Theorem~ \ref{charoftwoatomic-solution}.
   \begin{pro} \label{charoftwoatomic}
Suppose $\eta \in \nbb_2$, $\kappa = 1$, $p=2$ and
$\lambdab = \{\lambda_0\} \cup
\{\lambda_{i,1}\}_{i=1}^{\eta} \cup
\{\lambda_{i,2}\}_{i=1}^{\eta} \subseteq (0,\infty)$
are given. Then the following statements are
equivalent{\em :}
   \begin{enumerate}
   \item[(i)] $\lambdab$ has a subnormal completion $\slamh$ on
$\tcal_{\eta, 1}$ such that each measure $\mu_{i,
1}^{\lambdabh}$ is $2$-atomic,
   \item[(ii)] there exists a sequence
$\{r_i\}_{i=1}^{\eta} \subseteq (1,\infty)$ such that
   \begin{gather} \label{sum-0-n}
\sum_{i=1}^\eta r_i \frac{\lambda_{i,1}^2}
{\lambda_{i,2}^2} = 1,
   \\ \label{sum-01-n}
\sum_{i=1}^\eta r_i^2 \frac{\lambda_{i,1}^2}
{\lambda_{i,2}^4} < \frac{1}{\lambda_0^2},
   \end{gather}
   \item[(iii)] $-\infty < \beta(\eta) < \frac{1}{\lambda_0^2}$,
where\footnote{\;We adhere to the convention that
$\inf \varnothing = -\infty$; in particular, $-\infty
< \beta(\eta)$ means that the set in \eqref{bet-et} is
nonempty.}
   \begin{align} \label{bet-et}
\beta(\eta):=\inf \left\{\sum_{i=1}^\eta r_i^2
\frac{\lambda_{i,1}^2} {\lambda_{i,2}^4} \colon
\{r_i\}_{i=1}^{\eta} \subseteq (1,\infty) \text{ and }
\sum_{i=1}^\eta r_i \frac{\lambda_{i,1}^2}
{\lambda_{i,2}^2} = 1\right\}.
   \end{align}
   \end{enumerate}
   \end{pro}
%   \begin{proof}
%The implication (i)$\Rightarrow$(ii) can be deduced from
%its counterpart in Theorem~\ref{charoftwoatomic-infty}. The
%same is true for the ``moreover'' part. The equivalence
%(ii)$\Leftrightarrow$(iii) is easily seen to be true.
%Therefore, it remains to show that (ii) implies (i). Assume
%that (ii) holds. Since $\eta < \infty$, we infer from
%\eqref{sum-01-n} that there exists a constant $\vartheta
%\in (1,\infty)$ such that
%   \begin{align*}
%\vartheta \sum_{i=1}^\eta r_i^2 \frac{\lambda_{i,1}^2}
%{\lambda_{i,2}^4} < \frac{1}{\lambda_0^2}.
%   \end{align*}
%Setting $\vartheta_i=\vartheta$ for every $i\in J_{\eta}$,
%we verify that the sequences $\{r_i\}_{i=1}^{\eta},
%\{\vartheta_i\}_{i=1}^{\eta} \subseteq (1,\infty)$ satisfy
%the condition (ii) of Theorem~ \ref{charoftwoatomic-infty},
%so by this theorem (i) holds.
%   \end{proof}
Below we discuss the $2$-generation SCP on $\tcal_{\eta,1}$
with $2$-atomic measures under some constraints. Let us
suppose temporarily that $\eta \in \nbb_2$. According to
the ``moreover'' part of
Theorem~\ref{charoftwoatomic-infty}, if $\slamh$ is a
subnormal completion of $\lambdab$ on $\tcal_{\eta,1}$ such
that each measure $\mu_{i, 1}^{\lambdabh}$ is $2$-atomic,
then there exists $i\in J_{\eta}$ such that
$\sum_{j=1}^\eta \lambda_{j,1}^2 < \lambda_{i,2}^2$ (use
the fourth inequality in \eqref{kiki-1-infty}). If the last
inequality holds for all $i\in J_{\eta}$, then the solution
of the $2$-generation subnormal completion problem on
$\tcal_{\eta,1}$ with $2$-atomic measures takes a simple
form. Namely, the first inequality in \eqref{kiki-1-infty},
which is a necessary condition for solving the
$2$-generation subnormal completion problem on
$\tcal_{\eta,1}$, becomes now sufficient (see Corollary~
\ref{charoftwoatomic-2}(ii) below). It is worth pointing
out that in general the inequality $\lambda_0^2 \Le
\sum_{i=1}^{\eta} \lambda_{i,1}^2$ holds whenever the
$p$-generation subnormal completion problem on
$\tcal_{\eta,\kappa}$ has a solution (see \cite[Theorem~
4.9(i)]{EJSY1}).
   \begin{cor} \label{charoftwoatomic-2}
Suppose $\eta \in \nbb_2$, $\kappa = 1$, $p=2$ and
$\lambdab = \{\lambda_0\} \cup
\{\lambda_{i,1}\}_{i=1}^{\eta} \cup
\{\lambda_{i,2}\}_{i=1}^{\eta} \subseteq (0,\infty)$ are
given. Assume that
   \begin{align} \label{dodo-1}
\sum_{j=1}^\eta \lambda_{j,1}^2 < \lambda_{i,2}^2,
\quad i\in J_{\eta}.
   \end{align}
Then the following statements are equivalent{\em :}
   \begin{enumerate}
   \item[(i)] $\lambdab$ has a subnormal completion $\slamh$ on
$\tcal_{\eta, 1}$ such that each measure $\mu_{i,
1}^{\lambdabh}$ is $2$-atomic,
   \item[(ii)] $\lambda_0^2 < \sum_{i=1}^{\eta}
\lambda_{i,1}^2$.
   \end{enumerate}
   \end{cor}
   \begin{proof}
(i)$\Rightarrow$(ii) This is a direct consequence of the
first inequality in \eqref{kiki-1-infty}.

(ii)$\Rightarrow$(i) Define $\{r_i\}_{i=1}^{\eta}
\subseteq (1,\infty)$ by
   \begin{align} \label{choi-1}
r_i= \frac{\lambda_{i,2}^2}{\sum_{j=1}^\eta
\lambda_{j,1}^2}, \quad i\in J_{\eta}.
   \end{align}
It is easily seen that with this choice of
$\{r_i\}_{i=1}^{\eta}$ the equation \eqref{sum-0-n} holds.
Since
   \begin{align*}
\sum_{i=1}^\eta r_i^2 \frac{\lambda_{i,1}^2}
{\lambda_{i,2}^4} \overset{\eqref{choi-1}} =
\frac{1}{\sum_{i=1}^{\eta} \, \lambda_{i,1}^2}
\overset{\mathrm{(ii)}} < \frac{1}{\lambda_0^2},
   \end{align*}
we obtain \eqref{sum-01-n}. Applying
Proposition~\ref{charoftwoatomic} completes the proof.
   \end{proof}
   \begin{rem} \label{tas-1}
In view of Proposition~\ref{charoftwoatomic}(iii), if $\eta
\in \nbb_2$ one may create many sufficient conditions for
solving the $2$-generation subnormal completion problem on
$\tcal_{\eta,1}$ with $2$-atomic measures. The procedure
goes as follows. First, we look for a sequence
$\{r_i\}_{i\in J_{\eta}} \subseteq (1,\infty)$ satisfying
the equation \eqref{sum-0-n}. Second, we substitute this
particular choice of $\{r_i\}_{i\in J_{\eta}}$ into the
expression $\sum_{i=1}^\eta r_i^2 \frac{\lambda_{i,1}^2}
{\lambda_{i,2}^4}$ appearing in \eqref{sum-01-n}. Finally,
we require the value so obtained to be strictly less than~
$\frac{1}{\lambda_0^2}$.
   \end{rem}
It is worth mentioning that the procedure described in
Remark~ \ref{tas-1} has been applied in the proof of
Corollary~\ref{charoftwoatomic-2} (see \eqref{choi-1}).
Below we give a few more examples illustrating this
procedure.
   \begin{cor}
Suppose $\eta \in \nbb_2$, $\kappa = 1$, $p=2$ and
$\lambdab = \{\lambda_0\} \cup
\{\lambda_{i,1}\}_{i=1}^{\eta} \cup
\{\lambda_{i,2}\}_{i=1}^{\eta} \subseteq (0,\infty)$
are given. Then $\lambdab$ has a subnormal completion
$\slamh$ on $\tcal_{\eta, 1}$ such that each measure
$\mu_{i, 1}^{\lambdabh}$ is $2$-atomic provided any of
the following conditions holds{\em :}
   \begin{enumerate}
   \item[(i)] $\sum_{i=1}^\eta \frac{\lambda_{i,1}^2}
{\lambda_{i,2}^2} < 1$ and $\lambda_0^2 \,
\sum_{i=1}^\eta \frac{\lambda_{i,1}^2}
{\lambda_{i,2}^4} < \Big(\sum_{i=1}^\eta
\frac{\lambda_{i,1}^2} {\lambda_{i,2}^2}\Big)^2$,
   \item[(ii)] $\eta \lambda_{i,1}^2 < \lambda_{i,2}^2$
for each $i\in J_{\eta}$ and $\lambda_0^2 \,
\sum_{i=1}^\eta \frac{1} {\lambda_{i,1}^2} < \eta^2$,
   \item[(iii)] $\lambda_{i,1}^2 \sum_{j=1}^{\eta}
\frac{1}{\lambda_{j,2}^2} < 1$ for each $i\in
J_{\eta}$ and $\lambda_0^2 \sum_{i=1}^{\eta}
\frac{1}{\lambda_{i,1}^2\lambda_{i,2}^4} <
\Big(\sum_{i=1}^{\eta}
\frac{1}{\lambda_{i,2}^2}\Big)^2$.
   \end{enumerate}
   \end{cor}
   \begin{proof}
Apply Remark~ \ref{tas-1} to
   \begin{align*}
r_i = \frac{1}{\sum_{j=1}^{\eta}
\frac{\lambda_{j,1}^2} {\lambda_{j,2}^2}}, \quad r_i=
\frac{1}{\eta} \,
\frac{\lambda_{i,2}^2}{\lambda_{i,1}^2} \quad \text{
and } \quad r_i = \frac{1}{\lambda_{i,1}^2
\sum_{j=1}^{\eta} \frac{1}{\lambda_{j,2}^2}}
   \end{align*}
in the cases (i), (ii) and (iii), respectively.
   \end{proof}
The above discussion can be applied to solve the
$1$-generation subnormal completion problem on
$\tcal_{\eta,1}$ with $2$-atomic measures.
   \begin{cor}
Suppose $\eta \in \nbb_2$, $\kappa = 1$, $p=1$ and
$\lambdab = \{\lambda_0\} \cup
\{\lambda_{i,1}\}_{i=1}^{\eta} \subseteq (0,\infty)$
are given. Then the following statements are
equivalent{\em :}
   \begin{enumerate}
   \item[(i)] $\lambdab$ has a subnormal completion $\slamh$ on
$\tcal_{\eta, 1}$ such that each measure $\mu_{i,
1}^{\lambdabh}$ is $2$-atomic,
   \item[(ii)] $\lambda_0^2 < \sum_{i=1}^{\eta}
\lambda_{i,1}^2$.
   \end{enumerate}
   \end{cor}
   \begin{proof}
(i)$\Rightarrow$(ii) Let $\slamh$ be a subnormal completion
of $\lambdab$ on $\tcal_{\eta, 1}$ such that each measure
$\mu_{i, 1}^{\lambdabh}$ is $2$-atomic. Then clearly this
$\slamh$ is a subnormal completion of $\lambdab \cup
\{\hat\lambda_{i,2}\}_{i\in J_{\eta}}$ on $\tcal_{\eta,
1}$. Hence, by the condition \eqref{kiki-1-infty}, (ii) is
valid.

(ii)$\Rightarrow$(i) Choose any sequence
$\{\lambda_{i,2}\}_{i=1}^{\eta} \subseteq (0,\infty)$
that satisfies \eqref{dodo-1}. By Corollary~
\ref{charoftwoatomic-2}, $\lambdab \cup
\{\lambda_{i,2}\}_{i\in J_{\eta}}$ has a subnormal
completion $\slamh$ on $\tcal_{\eta, 1}$ such that
each measure $\mu_{i, 1}^{\lambdabh}$ is $2$-atomic.
As a consequence, this $\slamh$ is a subnormal
completion of $\lambdab$ on $\tcal_{\eta,1}$ as well.
This completes the proof.
   \end{proof}
In concluding this section, we consider the
$2$-generation subnormal completion problem on
$\tcal_{\eta,1}$ from the point of view of the
condition \eqref{dodo-1}.
   \begin{rem}
Suppose that $\eta\in \nbb_2$. Let us consider the
situation in which we are given initial data $\lambdab =
\{\lambda_0\} \cup \{\lambda_{i,1}\}_{i=1}^{\eta} \cup
\{\lambda_{i,2}\}_{i=1}^{\eta} \subseteq (0,\infty)$ on
$\tcal_{\eta, 1, 2}$ satisfying the equations
   \begin{align*}
\lambda_{i,2} = \lambda_{1,2}, \quad i \in J_\eta.
   \end{align*}
Notice that under the above assumption, if $\lambdab$
admits a subnormal completion on $\tcal_{\eta,1}$,
then it admits a $2$-generation flat subnormal
completion on $\tcal_{\eta,1}$ (see
\cite[Theorem~8.3]{EJSY1}), and any $2$-generation
flat subnormal completion $\slamh$ of $\lambdab$ on
$\tcal_{\eta,1}$ has the property that
$\mu_{i,1}^{\lambdabh}=\mu_{1,1}^{\lambdabh}$ for all
$i\in J_{\eta}$ (apply \eqref{IlB} to $u=e_{i,1}$).
Below, we discuss a few cases related to the condition
\eqref{dodo-1}.

\vspace{1ex}

$1^{\circ}$ If $\lambda_{1,2}^2 < \sum_{i=1}^\eta
\lambda_{i,1}^2$, then there is no completion of the
sequence
   \begin{align*}
\alpha = \left(\lambda_0, \sqrt{\sum_{i=1}^{\eta}
\lambda_{i,1}^2}, \lambda_{1,2} \right)
   \end{align*}
to a subnormal unilateral weighted shift, because it
is well known that a subnormal unilateral weighted
shift is hyponormal and so its weights must be
monotonically increasing. Therefore, by \cite[Theorem~
8.3(iii)]{EJSY1} there is no subnormal completion of
$\lambdab$ on $\tcal_{\eta, 1}$.

\vspace{1ex}

$2^{\circ}$ If $\lambda_{1,2}^2 = \sum_{i=1}^\eta
\lambda_{i,1}^2$, then the ``moreover'' part of
Theorem~\ref{charoftwoatomic-infty} shows that there
is no subnormal completion $\slamh$ of $\lambdab$ on
$\tcal_{\eta,1}$ with $2$-atomic measures
$\mu_{i,1}^{\lambdabh}$ since we lack the third
inequality in \eqref{kiki-1-infty}. Alternatively, we
may use the ``moreover'' part of \cite[Proposition~
7.6]{EJSY1}.

\vspace{1ex}

$3^{\circ}$ If $\lambda_{1,2}^2 > \sum_{i=1}^\eta
\lambda_{i,1}^2$ and $\lambda_0^2 < \sum_{i=1}^\eta
\lambda_{i,1}^2$ we may apply Corollary~
\ref{charoftwoatomic-2} to deduce that there is a
subnormal completion $\slamh$ of $\lambdab$ on
$\tcal_{\eta,1}$ with $2$-atomic measures
$\mu_{i,1}^{\lambdabh}$.

\vspace{1ex}

$4^{\circ}$ If $\lambda_{1,2}^2 > \sum_{i=1}^\eta
\lambda_{i,1}^2$ and $\lambda_0^2 = \sum_{i=1}^\eta
\lambda_{i,1}^2$ we may apply the condition
\eqref{kiki-1-infty} of
Theorem~\ref{charoftwoatomic-infty} to deduce that
there is no subnormal completion $\slamh$ of
$\lambdab$ on $\tcal_{\eta,1}$ with the
$\mu_{i,1}^{\lambdabh}$ $2$-atomic. Alternatively,
again using the ``moreover'' part of
\cite[Proposition~ 7.6]{EJSY1} we may deduce that the
measures $\mu_{i,1}^{\lambdabh}$ for any subnormal
completion $\slamh$ of $\lambdab$ on $\tcal_{\eta,1}$
(provided it exists) must be $1$-atomic.
   \end{rem}
   \section{\label{Sect5}An explicit solution of the \mbox{$2$}-generation SCP
on \mbox{$\tcal_{2,1}$}}
   In this section we give necessary and sufficient
conditions for solving the $2$-generation subnormal
completion problem on $\tcal_{2,1}$ with $2$-atomic
measures explicitly in terms of the initial data. In
view of Proposition~\ref{charoftwoatomic}(iii), the
solution of the $2$-generation subnormal completion
problem on $\tcal_{\eta,1}$ with $2$-atomic measures
reduces to computing the infimum $\beta(\eta)$ of the
quadratic form $\sum_{i=1}^\eta r_i^2
\frac{\lambda_{i,1}^2} {\lambda_{i,2}^4}$ subject to
the constraints that $\{r_i\}_{i=1}^{\eta} \subseteq
(1,\infty)$ and $\sum_{i=1}^\eta r_i
\frac{\lambda_{i,1}^2} {\lambda_{i,2}^2} = 1$. This
task is extremely complicated. The main difficulty
comes from the requirement that the numbers $r_i$
should belong to the open interval $(1,\infty)$. Even
in the case of $\eta=2$, there are three different
formulas for the infimum depending on weights in
question (see Theorem~\ref{iff-eta2} below).
%Our technique can be used to
%perform an explicit, though even more complicated, solution of
%the $2$-generation SCP on $\tcal_{\eta,1}$ with $\eta=3$.
However under some additional restrictive assumptions, the
infimum of the above quadratic form can be computed
explicitly.
   \begin{pro}
Let $\eta \in {\mathbb N}_{2}$, $\{a_{i}\}_{i=1}^{\eta}
\subseteq (0,\infty)$ and $\{b_{i}\}_{i=1}^{\eta} \subseteq
(0,\infty)$. Then
   \begin{equation} \label{min-1-a}
\min\bigg\{\sum_{i=1}^\eta b_{i} r_i^2 \colon
\{r_i\}_{i=1}^{\eta} \subseteq (0,\infty) \text{ and }
\sum_{i=1}^\eta a_{i} r_i =1\bigg\} =
\frac{1}{\sum_{i=1}^\eta \frac{a_{i}^2}{b_i}}.
   \end{equation}
Moreover, if $\sum_{j=1}^\eta \frac{a_{j}^2}{b_j} <
\frac{a_i}{b_i}$ for every $i\in J_{\eta}$, then
   \begin{equation} \label{min-2-a}
\min\bigg\{\sum_{i=1}^\eta b_{i} r_i^2 \colon
\{r_i\}_{i=1}^{\eta} \subseteq (1,\infty) \text{ and }
\sum_{i=1}^\eta a_{i} r_i =1\bigg\} =
\frac{1}{\sum_{i=1}^\eta \frac{a_{i}^2}{b_i}}.
   \end{equation}
In both cases the minimum is attained for
   \begin{align}  \label{min-2a}
r_i= \frac{1}{\frac{b_i}{a_i} \sum_{j=1}^\eta
\frac{a_{j}^2}{b_j}}, \quad i\in J_{\eta}.
   \end{align}
   \end{pro}
   \begin{proof}
Suppose $\{r_i\}_{i=1}^{\eta} \subseteq (0,\infty)$ and
$\sum_{i=1}^\eta a_{i} r_i =1$. It follows from the
Cauchy-Schwarz inequality that
   \begin{align*}
1 = \sum_{i=1}^\eta a_{i} r_i = \sum_{i=1}^\eta
\frac{a_{i}}{\sqrt{b_i}} \cdot \sqrt{b_i} \, r_i \Le
\bigg(\sum_{i=1}^\eta \frac{a_{i}^2}{b_i}\bigg)^{1/2}
\bigg(\sum_{i=1}^\eta b_{i} r_i^2\bigg)^{1/2}.
   \end{align*}
Therefore, we have
   \begin{align} \label{sum-2-a}
\sum_{i=1}^\eta b_{i} r_i^2 \Ge \frac{1}{\sum_{i=1}^\eta
\frac{a_{i}^2}{b_i}}.
   \end{align}
Substituting $\{r_i\}_{i=1}^{\eta}$ as in \eqref{min-2a}
shows that the inequality in \eqref{sum-2-a} becomes
equality. This proves \eqref{min-1-a}. The equation
\eqref{min-2-a} is a direct consequence of \eqref{min-1-a}.
   \end{proof}
As shown below, a solution of the $2$-generation
subnormal completion problem on $\tcal_{2,1}$ with
$2$-atomic measures can be written entirely in terms of
the initial data. This is done by computing the infimum
$\beta(2)$.
   \begin{thm} \label{iff-eta2}
Suppose $\eta =2$, $\kappa = 1$, $p=2$ and $\lambdab =
\{\lambda_0\} \cup \{\lambda_{i,1}\}_{i=1}^{2} \cup
\{\lambda_{i,2}\}_{i=1}^{2} \subseteq (0,\infty)$ are
given. Then $\lambdab$ has a subnormal completion $\slamh$
on $\tcal_{2, 1}$ with $2$-atomic measures $\mu_{1,
1}^{\lambdabh}$ and $\mu_{2, 1}^{\lambdabh}$ if and only if
$\beta(2) < \frac{1}{\lambda_0^2}$ and $1< \tau$, where
   \begin{align*}
\beta(2) =
   \begin{cases}
\frac{(a_1-a_2)^2+a_1b_1}{a_{2}^{2}b_{1}} & \text{if }
\sigma \Le 1,
   \\[2ex]
\frac{1}{a_1+b_1} & \text{if } 1 < \sigma < \tau,
   \\[2ex]
\frac{(b_2-b_1)^2+a_1b_1}{a_{1}b_{2}^2} & \text{if } \tau
\Le \sigma,
   \end{cases}
   \end{align*}
with $\sigma=\frac{a_2}{a_1+b_1}$, $\tau=
\frac{a_2(b_{2}-b_{1})}{a_{1}b_2}$, $a_j=\lambda_{1,j}^2$
and $b_j=\lambda_{2,j}^2$ for $j=1,2$.
   \end{thm}
   \begin{proof}
In view of Proposition~\ref{charoftwoatomic}, it is enough
to compute $\beta(2)$. If we replace $r_1$ and $r_2$ by $x$
and $y$, respectively, then \eqref{bet-et} with $\eta=2$
takes the form
   \begin{align*}
\beta(2)=\inf \Big\{\Theta(x,y) \colon x, y \in (1,\infty)
\text{ and } \phi (x,y) = 1\Big\},
   \end{align*}
where
   \begin{align*}
\Theta(x,y) = x^{2} \frac{a_{1}}{a_{2}^{2}} + y^{2}
\frac{b_{1}}{b_{2}^{2}} \quad \text{and} \quad \phi (x,y) =
x\frac{a_{1}}{a_{2}}+y\frac{b_{1}}{b_{2}} \quad \text{for}
\quad x,y\in (1,\infty).
   \end{align*}
If $x,y\in (1,\infty)$ are such that $\phi(x,y) = 1$, then
   \begin{align} \label{abcda}
y = \frac{b_2}{b_1} \Big(1-\frac{a_{1}}{a_{2}}x\Big)
>1,
   \end{align}
so $1< x < \tau$. As a consequence, we see that
$-\infty<\beta(2)$ if and only if $1 < \tau$. Substituting
$y$ as in \eqref{abcda} into $\Theta (x,y)$, we obtain
   \begin{align*}
\Theta (x,y)& = Ax^{2}-2Bx+C, \quad 1< x < \tau,
   \end{align*}
where
   \begin{align*}
\text{$A=\frac{a_1}{a_2^2}\left(1 + \frac{a_{1}}{b_{1}}
\right)$, $B=\frac{a_{1}}{a_{2}b_{1}}$ and
$C=\frac{1}{b_{1}}$.}
   \end{align*}
Note that $A,B,C>0$. The quadratic polynomial $\Omega
(x):=Ax^{2}-2Bx+C$ regarded as a function on $\rbb$ has a
minimum at $x=\frac{B}{A}$. Observe that
   \begin{align*}
\frac{B}{A}=\sigma \quad \text{and} \quad \Omega(\sigma)
=\frac{1}{a_1+b_1}>0.
   \end{align*}
It is now a routine matter to compute $\beta(2)$ by
considering three possible disjoint cases $\sigma \Le 1$,
$1 < \sigma < \tau$ and $\tau \Le \sigma$. What we get is
$\beta(2)=\varOmega(1)$ in the first case,
$\beta(2)=\varOmega(\sigma)$ in the second and
$\beta(2)=\varOmega(\tau)$ in the third one, where
   \begin{align*}
\varOmega(1) = \frac{(a_1-a_2)^2+a_1b_1}{a_{2}^{2}b_{1}}
\quad \text{and} \quad \varOmega(\tau) =
\frac{(b_2-b_1)^2+a_1b_1}{a_{1}b_{2}^2}.
   \end{align*}
This completes the proof.
   \end{proof}

   \vspace{1ex}

\textbf{Acknowledgement}. The authors take this opportunity
to express their appreciation both for the support of their
universities Bucknell University, Jagiellonian University
and Kyungpook National University materially aiding this
collaboration, and to the Departments of Mathematics at
which they have been guests in 2018 and 2019 for warm
hospitality.
   

\begin{thebibliography}{10}
%   \bibitem{Ag} J. Agler, Hypercontractions and subnormality, {\em J.
%Operator Theory}, {\bf 13}(1985), 203-217.
   \bibitem{BJJS2012} P. Budzy\'{n}ski, Z.J.
Jab{\l}o\'nski, I.B. Jung, J. Stochel, Unbounded
subnormal weighted shifts on directed trees, {\em J.
Math. Anal. Appl.} {\bf 394} (2012), 819-834.
   \bibitem{BJJS2013} P. Budzy\'{n}ski, Z.J.
Jab{\l}o\'nski, I.B. Jung, J. Stochel, Unbounded
subnormal weighted shifts on directed trees. II, {\em
J. Math. Anal. Appl.} {\bf 398} (2013), 600-608.
%   \bibitem{BJJS} P. Budzy\'{n}ski, Z. J. Jab{\l}o\'nski, I. B. Jung,
%J. Stochel, Unbounded subnormal composition operators
%in $L^2$-spaces, {\em J. Funct. Anal.} {\bf 269}
%(2015), 2110-2164.
   \bibitem{BJJS2018} P. Budzy\'{n}ski, Z.J.
Jab{\l}o\'nski, I.B. Jung, J. Stochel, {\em Unbounded
weighted composition operators in $L^2$-spaces}, Lect.
Notes Math., {\bf 2209}, Springer, 2018.
   \bibitem{Con1981} J.B. Conway, {\em Subnormal
operators}, Research Notes in Mathematics, {\bf 51}, Pitman
Publ. Co., London, 1981.
   \bibitem{Con1991} J.B. Conway, {\em The theory of
subnormal operators}, Mathematical Surveys and
Monographs, {\bf 36}, American Mathematical Society,
Providence, RI, 1991.
   \bibitem{CuF91} R.E. Curto, L.A. Fialkow,
Recursiveness, positivity, and truncated moment
problems, {\em Houston J. Math.} {\bf 17} (1991),
603-635.
   \bibitem{CuF93} R.E. Curto, L.A. Fialkow, Recursively generated
weighted shifts and the subnormal completion problem,
{\em Integr. Equ. Oper. Theory} {\bf 17} (1993),
202-246.
   \bibitem{CuF94} R.E. Curto, L.A. Fialkow, Recursively generated
weighted shifts and the subnormal completion problem
II, {\em Integr. Equ. Oper. Theory} {\bf 18} (1994),
369-426.
   \bibitem{CuF1996}
R. Curto, L. Fialkow, {\em Solution of the truncated
complex moment problem for flat data}, Mem. Amer.
Math. Soc. {\bf 119} (1996), no. 568, x+52 pp.
   \bibitem{EJSY1} G.R. Exner, I.B. Jung, J. Stochel,  H.Y. Yun,
A subnormal completion problem for weighted shifts on
directed trees, {\em Integr. Equ. Oper. Theory,} {\bf
90} (2018); https://doi.org/10.1007/s00020-018-2496-9.
    \bibitem{GW} R. Gellar,  L.J. Wallen,
Subnormal weighted shifts and the Halmos-Bram
criterion, {\em Proc. Japan Acad.} {\bf 46} (1970),
375-378.
   \bibitem{Hal1950} P.R. Halmos, Normal dilations and
extensions of operators, {\em Summa Brasil. Math.}
{\bf 2} (1950), 125-134.
    \bibitem{Hal} P.R. Halmos, Ten problems in Hilbert
space, {\em Bull. Amer. Math. Soc.} {\bf 76} (1970),
887-933.
   \bibitem{JJS2012} Z.J. Jab\l{o}\'nski, I.B.
Jung, J. Stochel, A non-hyponormal operator generating
Stieltjes moment sequences, {\em J. Funct. Anal.} {\bf
262} (2012), 3946-3980.
   \bibitem{JJS} Z.J. Jablonski, I.B. Jung,  J. Stochel,
Weighted shifts on diected trees, {\em Mem. Amer.
Math. Soc.} {\bf 216} (2012), viii+106 pp.
   \bibitem{JJKS2011}  Z.J. Jab{\l}o\'{n}ski,
I.B. Jung, J.A. Kwak, J. Stochel, Hyperexpansive
completion problem via alternating sequences; an
application to subnormality, {\em Linear Algebra
Appl.} {\bf 434} (2011), 2497-2526.
   \bibitem{Ore} O. Ore, {\em Theory of graphs}, American Mathematical
Society Colloquium Publications, Vol. XXXVIII,
American Mathematical Society, Providence, R.I. 1962.
   \bibitem{Shi} A.L. Shields, Weighted shift operators
and analytic function theory, {\em Topics in operator
theory}, pp. 49-128. Math. Surveys, No. 13, Amer.
Math. Soc., Providence, R.I., 1974.
   \bibitem{Sta} J.G. Stampfli,  Which weighted shifts are
subnormal, {\em Pacific J. Math.} {\bf 17} (1966),
367-379.
   \end{thebibliography}
   \end{document}